\documentclass[11pt,reqno]{amsart}

\usepackage{a4wide,enumerate,color,graphicx}
\usepackage{amsmath,bm}
\usepackage{scrextend} 
\allowdisplaybreaks
\usepackage[normalem]{ulem}
\usepackage[dvipsnames]{xcolor}
\usepackage{setspace}
\usepackage{amsbsy}
\usepackage{hyperref}

\newcommand{\eps}{\varepsilon}
\newcommand{\del}{\partial}
\newcommand{\diver}{\operatorname{div}}

\newtheorem{theorem}{Theorem}
\newtheorem{lemma}[theorem]{Lemma}

\numberwithin{equation}{section}
\hypersetup{
    colorlinks=true,
    linkcolor=blue,
    filecolor=blue,      
    urlcolor=blue,
}

\begin{document}

\title[Non-isothermal Multicomponent Flows]{Asymptotic Derivation of Multicomponent Compressible Flows with Heat Conduction and Mass Diffusion}

\author[S. Georgiadis]{Stefanos Georgiadis}
\address{Computer, Electrical and Mathematical Science and Engineering Division,
King Abdullah University of Science and Technology (KAUST),
Thuwal 23955-6900, Saudi Arabia and Institute for Analysis and Scientific Computing, Vienna University of Technology, Wiedner Hauptstra\ss e 8--10, 1040 Wien, Austria}
\email{stefanos.georgiadis@kaust.edu.sa} 

\author[A. Tzavaras]{Athanasios E. Tzavaras}
\address{Computer, Electrical and Mathematical Science and Engineering Division,
King Abdullah University of Science and Technology (KAUST),
Thuwal 23955-6900, Saudi Arabia}
\email{athanasios.tzavaras@kaust.edu.sa}

\begin{abstract}
A Type-I model of a multicomponent system of fluids with non-constant temperature is derived as the high-friction limit of a Type-II model via a Chapman-Enskog expansion. The asymptotic model is shown to  fit into the general theory of hyperbolic-parabolic systems, by exploiting the entropy structure inherited through the asymptotic procedure. Finally, by deriving the relative entropy identity for the Type-I model, two convergence results for smooth solutions are presented, from the system with mass-diffusion and heat conduction to the corresponding system without mass-diffusion but including heat conduction and to its hyperbolic counterpart.
\end{abstract}

\date{\today}

\maketitle


\section{Introduction}

Multicomponent systems of fluids occur often in nature and in industry: the Earth's atmosphere consists of nitrogen, oxygen, argon, carbon dioxide and small amounts of other gases. Natural gas is made of gaseous hydrocarbons such as methane, ethane, propane. The widespread presence of multicomponent fluids suggests the importance in understanding their modeling and being able to predict their behavior. 

Depending on the phenomenon and the application, models might vary containing more or less modeling detail. While one might prefer the more detailed model, it is conceivable that modeling detail might not be available or that it cannot be experimentally measured. There should be a balance: a theory must be detailed enough to sufficiently describe a phenomenon, but not so detailed that significant phenomena are hard to comprehend \cite{DP}. Modeling of multicomponent fluids is a well-developed subject and the reader is referred to \cite{Mue,MR,DP,GioMulti,Rug,BDre} for various approaches. 

In a multicomponent theory with primitive variables mass density, velocity and temperature, one distinguishes among three classes of models: In a Type-I model, each component is described by its own mass density, but the components move with a common velocity and have a common temperature. In Type-II models, each component is described by its own mass density and velocity, but the components have a common temperature. 
Type-III models (which are not considered here) are described via the individual densities, velocities and temperatures of each component.
For information regarding Type-III models we refer to \cite{Rug}, while for a short discussion on motivations and reasons to employ each model we refer to \cite{BDre}.

It is useful to have systematic ways of passing from a detailed theory to a less detailed one, which is mathematically easier to handle and experimentally easier to measure. 
Previous works trying to pass from a Type-II to a Type-I model are already known, for example in \cite{BDre} in which the authors investigate the reduction of a non-isothermal model using an entropy invariant method, or in \cite{HJT}, in which the reduction is done using asymptotic methods, but for the isothermal case and for a simple mixture (see Appendix \ref{appB}). The first objective of this paper is to derive a Type-I model using the asymptotic method of \cite{HJT}, but for the non-isothermal, non-simple mixture model treated in \cite{BDre}.

To this end, consider the Type-II system of equations for multicomponent fluids: 
\begin{equation}\label{intro-eq1}
    \del_t\rho_i+\textnormal{div}(\rho_iv_i)=0
\end{equation}
 \begin{equation}\label{intro-eq2}
   \del_t (\rho_iv_i)+\textnormal{div}(\rho_iv_i\otimes v_i)=\rho_ib_i-\rho_i\nabla\mu_i-\frac{1}{\theta}(\rho_ie_i+p_i-\rho_i\mu_i)\nabla\theta-\frac{\theta}{\epsilon}\sum_{j\not=i}b_{ij}\rho_i\rho_j(u_i-u_j)
\end{equation} 
\begin{equation}\label{intro-eq3}
\begin{split}
   &\del_t  \left(\rho e+\sum_i\frac{1}{2}\rho_iv_i^2\right) +\textnormal{div}\left(\left(\rho e+\sum_i\frac{1}{2}\rho_iv_i^2\right)v\right) 
   \\
   &\quad = \textnormal{div}\left(\kappa\nabla\theta-\sum_i(\rho_ie_i+p_i+\frac{1}{2}\rho_iv_i^2)u_i\right) 
        - \textnormal{div}(pv)+\rho b\cdot v+\rho r+\sum_i\rho_ib_i\cdot u_i
\end{split}
\end{equation}
Equations \eqref{intro-eq1} are the partial mass balances and \eqref{intro-eq2} the partial momentum balances of the $n$ components of the fluid, while \eqref{intro-eq3} is the balance of total energy of the mixture. The index $i\in\{1,\dots,n\}$ refers to the $i$-th component of the fluid. The prime variables are the mass densities $\rho_i$, the partial velocities $v_i$ and the temperature of the mixture $\theta$. We define the total mass $\rho=\sum_i\rho_i$, the barycentric velocity of the mixture $v := \frac{1}{\rho} \sum_i\rho_iv_i$ and the diffusional velocities $u_i :=v_i-v$;
the latter satisfy $\sum_i\rho_iu_i=0$. Moreover, we have the following thermodynamic quantities: the chemical potentials $\mu_i$, the partial pressures $p_i$, the specific internal energies $e_i$ and we define the total pressure $p=\sum_ip_i$ and the thermal energy $\rho e=\sum_i\rho_ie_i$. Furthermore, $b_i$ stands for the body forces acting on the $i$-th component and $r_i$ for the heat supply due to radiation, where $\rho b=\sum_i\rho_ib_i$ is the total body force and $\rho r=\sum_i\rho_ir_i$ the total heat supply due to radiation. Finally, $\kappa=\kappa(\rho_1,\dots,\rho_n,\theta)\ge0$ is the heat conductivity and $b_{ij}=b_{ij}(\rho_i,\rho_j,\theta)$ are strictly positive and symmetric coefficients that model the interaction between the $i$-th and $j$-th components with a strength that is measured by $\epsilon>0$.

The system is complemented by a set of constitutive relations, which read \begin{equation} \label{intro-const1}
    \rho\psi=\rho\psi(\rho_1,\dots,\rho_n,\theta)
\end{equation} \begin{equation} \label{intro-const2}
    \rho\psi=\rho e-\rho\eta\theta
\end{equation}
\begin{equation} \label{intro-const3}
    (\rho\psi)_{\rho_i}=\mu_i
\end{equation} \begin{equation} \label{intro-const4}
    (\rho\psi)_\theta=-\rho\eta
\end{equation} where $\rho\psi$ is the Helmholtz free-energy and $\rho\eta$ the entropy, along with the Gibbs-Duhem relation, for determining the total pressure: 
\begin{equation} \label{intro-gd}
    \rho\psi+p=\sum_i\rho_i\mu_i
\end{equation} as explained in the two appendices. The above type-II model is proposed in \cite{BDre}. The format of equations \eqref{intro-const1}-\eqref{intro-const4} is
motivated by the usual considerations of equilibrium thermodynamics, while the model is consistent with the Clausius-Duhem inequality. 
For the reader's convenience we present an outline of its thermodynamic structure in Appendices \ref{appA} and \ref{appB}, while for further elaboration and detail we refer to \cite{DP,GioMulti,BDre}.

In addition to system \eqref{intro-eq1}-\eqref{intro-eq3}, consider also the system 
\begin{align}
\label{intro-eq4}
    \del_t\rho_i +\textnormal{div}(\rho_iv) &=-\textnormal{div}(\epsilon\rho_i\tilde{u}_i)
\\
\label{intro-eq5}
    \del_t(\rho v)+\textnormal{div}(\rho v\otimes v) &=\rho b-\nabla p
\\
\label{intro-eq6}
\begin{split}
   \del_t \left(\rho e+\frac{1}{2}\rho v^2\right)+\textnormal{div}\left((\rho e+\frac{1}{2}\rho v^2)v\right) &= \textnormal{div}\left(\kappa\nabla\theta-\epsilon\sum_i(\rho_ie_i+p_i)\tilde{u}_i\right) \\
    &\quad - \textnormal{div}(pv)+\rho r+\rho b\cdot v+\sum_i\rho_ib_i\cdot u_i
\end{split}
\end{align}
where $u_i=\epsilon\tilde{u}_i$ is determined by solving the constrained algebraic system of Maxwell-Stefan type
\begin{equation}
\label{intro-eq7}
\begin{aligned}
    -\sum_{j\not=i}b_{ij}\theta\rho_i\rho_j(u_i-u_j) &=\epsilon\left(-\frac{\rho_i}{\rho}\nabla p+\rho_i\theta\nabla\frac{\mu_i}{\theta}-\theta(\rho_ie_i+p_i)\nabla\frac{1}{\theta}\right)
    \\
    \sum_i \rho_i u_i &= 0
\end{aligned} 
\end{equation}
The system \eqref{intro-eq4}-\eqref{intro-eq7} forms a Type-I model, with the same notation as before and the same constitutive relations. The mathematical structure of multicomponent systems falls into the general realm of hyperbolic-parabolic systems \cite{GM,BGP} and follows the general
framework originated by Kawashima \cite{Kaw}. One objective here is to study their structure from the perspective of relative entropy identities 
and the general theory in \cite{CT}.

Analytical results on multi-component flows extend in various directions.
We refer to \cite{Bot,JS,MaxS} for  existence and uniqueness studies of strong and/or weak solutions 
for Maxwell-Stefan systems of mass diffusion, situations that involve no mean flow and pertain to the general area of parabolic systems.
By contrast, situations that involve mean flow lead to questions in the realm of hyperbolic or hyperbolic-parabolic systems.
There are available analyses for isothermal viscous flows of multicomponent systems \cite{BDru} and even for multicomponent compressible Euler flows  \cite{PSZMulti,OR,PSZTwo,YZ}.  For analyses of non-isothermal multicomponent systems that include effects of heat-conduction we refer to \cite{MPZConstr,ZatMix,MPZExi,BJPZ}. The above works concern the mathematical structure of multi-component systems and existence of
solutions for steady and dynamic problems. There has been recent interest in the convergence from compressible multicomponent Euler 
equations in the high-friction limit to Maxwell-Stefan systems, a problem pertaining to the subject of relaxation approximations.
A number of studies have appeared regarding isothermal flows \cite{BGP,GYY,HJT,OR} achieving in the limit the classical Maxwell-Stefan system \cite{BGP},
 or (for more general chemical potentials) porous media variants or even fourth order diffusions \cite{HJT, OR}.

The focus of the present work is on the system \eqref{intro-eq4}-\eqref{intro-eq7} modeling non-isothermal multi-component flows 
that include the effects of mass-diffusion and heat conduction but no viscous effects.  As is typical for Maxwell-Stefan systems, a key difficulty arises from the inversion
of \eqref{intro-eq7}. First, the Chapman-Enskog analysis from the isothermal case \cite{GYY,HJT} is extended to the non-isothermal case
thus obtaining an asymptotic derivation of the Type-I model from the Type-II model \eqref{intro-eq1}-\eqref{intro-eq3}. 
Next,  system \eqref{intro-eq1}-\eqref{intro-eq3} is complemented with its corresponding entropy identity \eqref{eq-entropytypeII} and, by 
performing an asymptotic analysis  to the entropy identity, we obtain the entropy identity 
\eqref{eq-entropytypeI} for the system  \eqref{intro-eq4}-\eqref{intro-eq7}. This analysis indicates how the emerging system inherits 
the dissipative structure of the original system.

 The second step is to verify that \eqref{intro-eq4}-\eqref{intro-eq7} fits into the general framework of systems of hyperbolic-parabolic type as introduced in \cite{Kaw} and generalized in \cite{CT}.
 There are two aspects to this question: (i) to address the connection between the "mathematical entropy" and the thermodynamic structure of the model
 and (ii) to identify the dissipative structure of mass-diffusion and heat conduction.
The latter connects to the issue of inversion of the constrained algebraic system \eqref{intro-eq7}.
This is  overcome by using the notion of the Bott-Duffin inverse of a matrix which was brought to the subject of Maxwell-Stefan systems in \cite{MaxS};
the latter provides an important ingredient for  inverting the algebraic system and comparing the entropic structure of the Type-I model 
with the usual entropic structure of hyperbolic-parabolic systems in the format discussed in \cite{Kaw,CT}.

Third, we derive a relative entropy identity for \eqref{intro-eq4}-\eqref{intro-eq7} using the methodological approach of relative entropy \cite{Daf,iesan,CT}.
We compute the relative entropy identity \eqref{conv-eq15} which monitors the evolution of \eqref{relen} and is remarkable in its simplicity.
This identity is, in turn, used in order to prove convergence from strong solutions of the system \eqref{intro-eq4}-\eqref{intro-eq7} to strong solutions of heat-conducting
multicomponent Euler system when the mass-diffusivity $\epsilon$ tends to zero. Also to prove convergence to smooth solutions
of multi-component adiabatic Euler flows when both heat conductivity $\kappa$ and mass diffusivity $\epsilon$ tend to zero.

The outline of this work is as follows.
In section \ref{sec2}, we derive \eqref{intro-eq4}-\eqref{intro-eq7} as the high-friction limit of system \eqref{intro-eq1}-\eqref{intro-eq3}, by doing a Chapman-Enskog expansion,
and present the asymptotic derivation of the entropy identity \eqref{eq-entropytypeI}.
In section \ref{sec3}, we verify that \eqref{intro-eq4}-\eqref{intro-eq7} fits into the general framework of systems of hyperbolic-parabolic type  \cite{Kaw,CT} and show that the system without mass diffusion and heat conduction, obtained by setting $\epsilon=\kappa=0$, is hyperbolic, by computing its wave speeds. 
In section \ref{sec5}, using the methodological approach of relative entropy \cite{Daf,iesan,CT}, we derive the relative entropy identity \eqref{conv-eq15}, which is used
to prove convergence from strong solutions of \eqref{intro-eq4}-\eqref{intro-eq7} to strong solutions of the heat-conducting multicomponent Euler flows
as $\epsilon \to 0$ or to adiabatic multicomponent Euler flows when both $\epsilon, \kappa \to 0$.
We conclude with four appendices: the first two appendices present the natural modeling framework of multicomponent Type-II models
 \eqref{intro-eq1}-\eqref{intro-eq3}, outline its thermodynamical structure and the results of \cite{BDre} on the consistency of the model with the Clausius-Duhem inequality;  the third appendix provides the computations for deriving the entropy equation of system \eqref{intro-eq1}-\eqref{intro-eq3}, while the fourth one the computation for the relative entropy identity.


\section{Chapman-Enskog Expansion of the Type-II Model}
\label{sec2}

The goal of the asymptotic procedure is to obtain from \eqref{intro-eq1}-\eqref{intro-eq3} a Type-I model via a Chapman-Enskog expansion that will provide an approximation of the initial system of order $\epsilon^2$. The resulting system will consist of $n$ partial momentum balances, a single (total) momentum balance and one total energy balance, thus the name Type-I. It will also be complemented by the same constraint $\sum_i\rho_iu_i=0$ (at least up to order $\mathcal{O}(\epsilon^2)$) and a linear system for determining the diffusional velocities $u_i$.

The resulting system should contain no partial velocities. Using the definition of the diffusional velocities $u_i=v_i-v$ we perform the change of variables $(v_1,\dots,v_n)\mapsto(v,u_1,\dots,u_n)$. Then system \eqref{intro-eq1}-\eqref{intro-eq3} reads: \begin{equation}
\label{cee-eq1}
   \del_t \rho_i +\textnormal{div}(\rho_iu_i+\rho_iv)=0
\end{equation} \begin{equation}
\label{cee-eq2}
\begin{split}
    \del_t \left(\rho_i(u_i+v)\right)+\textnormal{div}(\rho_i(u_i+v)\otimes(u_i+v))=\rho_ib_i & - \rho_i\nabla\mu_i-\frac{1}{\theta}(\rho_ie_i+p_i-\rho_i\mu_i)\nabla\theta \\
    & -\frac{\theta}{\epsilon}\sum_{j\not=i}b_{ij}\rho_i\rho_j(u_i-u_j)
\end{split}
\end{equation} \begin{equation}
\label{cee-eq3}
    \begin{split}
        \del_t\Big(\rho e & + \sum_i\frac{1}{2}\rho_i(u_i+v)^2\Big)+\textnormal{div}\left(\Big(\rho e+\sum_i\frac{1}{2}\rho_i(u_i+v)^2\Big)v\right)=-\textnormal{div}(pv)+\rho b\cdot v \\
        & + \textnormal{div}\left(\kappa\nabla\theta-\sum_i(\rho_ie_i+p_i+\frac{1}{2}\rho_i(u_i+v)^2)u_i\right) +\rho r+\sum_i\rho_ib_i\cdot u_i
    \end{split}
\end{equation}  subject to the constraint \begin{equation}\label{cee-eq4}
    \sum_i\rho_iu_i=0
\end{equation}

\subsection{Hilbert Expansion}
\label{sec2.1}
Observe that, as $\epsilon \to 0$, we formally obtain the equation 
$$
\displaystyle{-\theta \sum_{j\not=i}b_{ij}\rho_i\rho_j(u_i-u_j) = 0} \, .
$$
In \cite{HJT} it was proved that the system \begin{equation} \label{cee-eq19}
    -\sum_{j\not=i}B_{ij}(u_i-u_j)=d_i, ~~ i=1,\dots,n~~\textnormal{subject to}~~\sum_{i}\rho_iu_i=0
\end{equation} under the hypotheses: 

(i) $d_1,\dots,d_n\in\mathbb{R}^3$ satisfy $\sum_id_i=0$

(ii) $\rho_1,\dots,\rho_n>0$

(iii) $(B_{ij})\in\mathbb{R}^{n\times n}$ is a symmetric matrix with $B_{ij}\ge0$ for $i,j=1,\dots,n$ and

(iv) all solutions of the homogeneous system $\sum_{j\not=i}B_{ij}(u_i-u_j)=0$, $i=1,\dots,n$ lie in the space spanned by $(1,\dots,1)\in\mathbb{R}^n$, 
\newline
has the unique solution \begin{align} \label{cee-eq20}
    \rho_iu_i = -\sum_{j,k=1}^{n-1}(\delta_{ij}\rho_i & - \frac{\rho_i\rho_j}{\rho})\tau_{jk}^{-1}d_k,~~i=1,\dots,n-1 \\
    \rho_nu_n & = -\sum_{j=1}^{n-1}\rho_ju_j \nonumber
\end{align} where $(\tau_{ij}^{-1})\in\mathbb{R}^{(n-1)\times(n-1)}$ is the inverse of a regular submatrix, obtained by reordering the matrix $(\tau_{ij})\in\mathbb{R}^{n\times n}$ of rank $n-1$ with coefficients \[ \tau_{ij}=\delta_{ij}\sum_{k=1}^nB_{ik}-B_{ij} \, .
 \] 
 Our case is simpler since the coefficients $b_{ij}$ are symmetric and strictly positive. We also require that $\theta>0$, which means that hypotheses (iii) and (iv) are automatically satisfied for $\rho_i>0$, $i=1,\dots,n$.
 
Next,  introduce the Hilbert expansion \begin{equation} \label{cee-eq5}
    \rho_i=\rho_i^0+\epsilon\rho_i^1+\epsilon^2\rho_i^2+\mathcal{O}(\epsilon^3)
\end{equation} \begin{equation} \label{cee-eq6}
    u_i=u_i^0+\epsilon u_i^1+\epsilon^2u_i^2+\mathcal{O}(\epsilon^3)
\end{equation}\begin{equation} \label{cee-eq7}
    v=v^0+\epsilon v^1+\mathcal{O}(\epsilon^2)
\end{equation}\begin{equation} \label{cee-eq8}
    \theta=\theta^0+\epsilon\theta^1+\epsilon^2\theta^2+\mathcal{O}(\epsilon^3)
\end{equation} Inserting this into $\rho=\sum_i\rho_i$ and the constraint \eqref{cee-eq4} we obtain \begin{equation} \label{cee-eq9}
    \rho=\rho^0+\epsilon\rho^1+\mathcal{O}(\epsilon^2)
\end{equation} where we set $\rho^0:=\sum_i\rho_i^0$ and $\rho^1:=\sum_i\rho_i^1$ and \begin{equation} \label{cee-eq10}
    \sum_i\rho_i^0u_i^0+\epsilon\sum_i(\rho_i^1u_i^0+\rho_i^0u_i^1)+\mathcal{O}(\epsilon^2)=0
\end{equation} Equating terms of the same order gives \begin{equation} \label{cee-eq11}
    \sum_i\rho_i^0u_i^0=0~~\textnormal{and}~~\sum_i(\rho_i^1u_i^0+\rho_i^0u_i^1)=0
\end{equation} Next we insert the expansion \eqref{cee-eq5}-\eqref{cee-eq8} into system \eqref{cee-eq1}-\eqref{cee-eq3} and identify terms of the same order:

(i) The terms at the order $\mathcal{O}(1/\epsilon)$ give:
\begin{equation} \label{cee-eq12}
    -\theta^0\sum_{j\not=i}b_{ij}\rho_i^0\rho_j^0(u_i^0-u_j^0)=0
\end{equation} 
For the expansion we assume $\rho^0_i > 0$ and note that  \eqref{cee-eq12} is of the form \eqref{cee-eq19} with $d_i=0$ and $B_{ij}$ symmetric and strictly positive, which implies that the null-space of the homogeneous system is $\textnormal{span}\{(1,\dots,1)\}$. This, along with \eqref{cee-eq20}, gives $u_i^0=0$ for $i\in\{1,\dots,n\}$. This is incorporated in the remaining expansions. 

For the expansion of $\frac{1}{\theta}$ we use the Taylor series $\frac{1}{1+x}=1-x+x^2+\mathcal{O}(x^3)$ and thus 
\begin{equation} \label{eq-exp1theta}
    \frac{1}{\theta} = \frac{1}{\theta^0(1+\epsilon\frac{\theta^1}{\theta^0}+\epsilon^2\frac{\theta^2}{\theta^0}+\cdots)} 
    = \frac{1}{\theta^0}-\epsilon\frac{\theta^1}{(\theta^0)^2}+\epsilon^2\frac{(\theta^1)^2-\theta^2\theta^0}{(\theta^0)^3}+\mathcal{O}(\epsilon^3)
\end{equation}
For the expansions of the thermodynamic functions we use
\begin{equation}\label{expansione}
\begin{aligned}
e_i &= \hat e_i ( \rho_1^0+\epsilon\rho_1^1 +\mathcal{O}(\epsilon^2) , ... ,  \rho_n^0+\epsilon\rho_n^1+ \mathcal{O}(\epsilon^2)  ,  \theta^0+\epsilon\theta^1+\mathcal{O}(\epsilon^2))
\\
&= e_i^0 + \epsilon e_i^1 + \epsilon^2 e_i^2 + \mathcal{O}(\epsilon^3) 
\end{aligned}
\end{equation}
where $e_i^0, \, e_i^1 , \, e_i^2 , ...$ are computed by using the Taylor theorem, for instance
$$
\begin{aligned}
e_i^0 &= \hat e_i (\rho_1^0,  ..., \rho_n^0 , \theta^0) \, ,  \quad 
e_i^1 = \sum_{j=1}^n  \frac{\del \hat e_i}{\del \rho_j} (\rho_1^0, ..., \rho_n^0, \theta^0) \, \rho_j^1 +  \frac{\del \hat e_i}{\del \theta} (\rho_1^0, ..., \rho_n^0, \theta^0)  \, \theta^1  \, ,
\\
e_i^2 &= \sum_{j=1}^n  \frac{\del \hat e_i}{\del \rho_j} (\rho^0, \theta^0) \, \rho_j^2 +  \frac{\del \hat e_i}{\del \theta} (\rho^0, \theta^0)  \, \theta^2
             + \sum_{k,l=1}^n \frac{\del^2 \hat e_i}{\del \rho_k \del \rho_l} (\rho^0, \theta^0) \, \rho_k^1 \rho_l^1
             \\
        &\quad  + \sum_{k=1}^n \frac{\del^2 \hat e_i}{\del \rho_k \del \theta} (\rho^0, \theta^0) \, \rho_k^1 \theta^1
        + \frac{\del^2 \hat e_i}{\del \theta^2} (\rho^0, \theta^0) \, (\theta^1)^2
\end{aligned}
$$
and so on. A similar expansion is used for the entropy
\begin{equation}\label{expansioneta}
\begin{aligned}
\eta &= \hat \eta \left  (    \rho_1^0+\epsilon\rho_1^1 +\mathcal{O}(\epsilon^2) , ... ,  \rho_n^0+\epsilon\rho_n^1+ \mathcal{O}(\epsilon^2)  ,  \theta^0+\epsilon\theta^1+\mathcal{O}(\epsilon^2)  \right )
\\
&= \eta^0 + \epsilon \eta^1 + \epsilon^2 \eta^2 + \mathcal{O}(\epsilon^3) 
\end{aligned}
\end{equation}
as well as for the functions $e_i$, $p_i$ and $\kappa$ all of which only depend on $(\rho_1, ... , \rho_n)$ and $\theta$.
We then obtain:

(ii) Terms of order $\mathcal{O}(1)$: \begin{equation} \label{cee-eq21}
    \del_t\rho_i^0+\textnormal{div}(\rho_i^0v^0)=0
\end{equation} 
\begin{equation} \label{cee-eq22}
\begin{split}
    \del_t(\rho_i^0v^0)+\textnormal{div}(\rho_i^0v^0\otimes v^0)=\rho_i^0b_i^0-\rho_i^0\nabla\mu_i^0 & - \frac{1}{\theta^0}(\rho_i^0e_i^0+p_i^0-\rho_i^0\mu_i^0)\nabla\theta^0 \\
    & - \theta^0\sum_{j\not=i}b_{ij}\rho_i^0\rho_j^0(u_i^1-u_j^1)
\end{split}
\end{equation} 
\begin{equation} \label{cee-eq23}
\begin{split}
    \del_t\Big(\sum_i\rho_i^0e_i^0 & + \sum_i\frac{1}{2}\rho_i^0(v^0)^2\Big)+\textnormal{div}\left(\left(\sum_i\rho_i^0e_i^0+\sum_i\frac{1}{2}\rho_i^0(v^0)^2\right)v^0\right) \\
    & = -\textnormal{div}(p^0v^0)+\sum_i\rho_i^0r_i^0+\sum_i\rho_i^0b_i^0\cdot v^0+\textnormal{div}\left(\kappa^0\nabla\theta^0\right)
\end{split}
\end{equation} 

(iii) Terms of order $\mathcal{O}(\epsilon)$: \begin{equation} \label{cee-eq24}
    \del_t\rho_i^1+\textnormal{div}(\rho_i^0(u_i^1+v^1)+\rho_i^1v^0)=0
\end{equation} \begin{equation} \label{cee-eq25}
\begin{split}
    & \del_t(\rho_i^1v^0+\rho_i^0(u_i^1+v^1))+\textnormal{div}\Big(\rho_i^1v^0\otimes v^0+\rho_i^0(u_i^1+v^1)\otimes v^0+\rho_i^0v^0\otimes(u_i^1+v^1)\Big) \\
    & = \rho_i^1b_i^0+\rho_i^0b_i^1-\rho_i^1\nabla\mu_i^0-\rho_i^0\nabla\mu_i^1+\frac{\theta^1}{(\theta^0)^2}(\rho_i^0e_i^0+p_i^0-\rho_i^0\mu_i^0)\nabla\theta^0 \\
    & - \frac{1}{\theta^0}(\rho_i^1e_i^0+\rho_i^0e_i^1+p_i^1-\rho_i^1\mu_i^0-\rho_i^0\mu_i^1)\nabla\theta^0-\frac{1}{\theta^0}(\rho_i^0e_i^0+p_i^0-\rho_i^0\mu_i^0)\nabla\theta^1 \\
    & - \theta^0\sum_{j\not=i}b_{ij}\Big\{(\rho_i^1\rho_j^0+\rho_i^0\rho_j^1)(u_i^1-u_j^1)+\rho_i^0\rho_j^0(u_i^2-u_j^2)\Big\}-\theta^1\sum_{j\not=i}b_{ij}\rho_i^0\rho_j^0(u_i^1-u_j^1)
\end{split}
\end{equation} \begin{equation} \label{cee-eq26}
\begin{split}
    &\del_t\Big(\sum_i(\rho_i^0e_i^1+\rho_i^1e_i^0)+\sum_i\frac{1}{2}\rho_i^1(v^0)^2+\sum_i\frac{1}{2}\rho_i^02v^0(u_i^1+v^1)\Big)+\textnormal{div}\Big(\Big(\sum_i\rho_i^0e_i^0 \\
    & + \sum_i\frac{1}{2}\rho_i^0(v^0)^2\Big)v^1+\Big(\sum_i(\rho_i^0e_i^1+\rho_i^1e_i^0)+\sum_i\frac{1}{2}\rho_i^02v^0(u_i^1+v^1)+\sum_i\frac{1}{2}\rho_i^1(v^0)^2\Big)v^0\Big) \\
    & = -\textnormal{div}(p^0v^1+p^1v^0)+\sum_i(\rho_i^0b_i^0\cdot v^1+(\rho_i^1b_i^0+\rho_i^0b_i^1)\cdot v^0)+\sum_i\rho_i^0b_i^0\cdot u_i^1+\sum_i(\rho_i^0r_i^1+\rho_i^1r_i^0) \\
    & + \textnormal{div}\left(\kappa^0\nabla\theta^1+\kappa^1\nabla\theta^0-\sum_i\left\{(\rho_i^0e_i^0+p_i^0)+\frac{1}{2}\rho_i^0(v^0)^2\right\}u_i^1\right)
\end{split}
\end{equation}  Summing \eqref{cee-eq22} over $i$ and using the symmetry of $b_{ij}$ we obtain \begin{equation} \label{cee-eq27}
    \del_t(\sum_i\rho_i^0v^0)+\textnormal{div}(\sum_i\rho_i^0v^0\otimes v^0)=\sum_i\rho_i^0b_i^0-\nabla p^0
\end{equation} The reason behind the simplification of the right-hand side comes from the thermodynamics relations of the problem. More precisely, we have: \begin{equation} \label{cee-eq28}
    \begin{split}
    \sum_i\left(\rho_ib_i-\rho_i\nabla\mu_i-\frac{1}{\theta}(\rho_ie_i+p_i-\rho_i\mu_i)\nabla\theta\right) 
    & = \rho b-\sum_i\rho_i\nabla\mu_i-\frac{1}{\theta}(\rho e+p-\sum_i\rho_i\mu_i)\nabla\theta \\
    & \overset{\eqref{intro-gd}}{=} \rho b-\sum_i\rho_i\nabla\mu_i-\rho\eta\nabla\theta \\
    & = \rho b+\sum_i\mu_i\nabla\rho_i-\nabla\left(\sum_i\rho_i\mu_i\right)-\rho\eta\nabla\theta \\
    & \overset{\eqref{intro-const3}-\eqref{intro-gd}}{=} \rho b-\nabla p
\end{split}
\end{equation} For more details on the thermodynamic relations used in the calculation above, we refer to Appendices \ref{appA} and \ref{appB}. 
Expanding both sides of \eqref{cee-eq28}, we get: 
\begin{equation} \label{cee-eq29}
    \sum_i\left(\rho_i^0b_i^0-\rho_i^0\nabla\mu_i^0-\frac{1}{\theta^0}(\rho_i^0e_i^0+p_i^0-\rho_i^0\mu_i^0)\nabla\theta^0\right)=\sum_i\rho_i^0b_i^0-\nabla p^0
\end{equation} and \begin{equation} \label{cee-eq30}
\begin{split}
    \sum_i\Big(\rho_i^1b_i^0 & + \rho_i^0b_i^1-\rho_i^1\nabla\mu_i^0-\rho_i^0\nabla\mu_i^1+\frac{\theta^1}{(\theta^0)^2}(\rho_i^0e_i^0+p_i^0-\rho_i^0\mu_i^0)\nabla\theta^0-\frac{1}{\theta^0}(\rho_i^1e_i^0+\rho_i^0e_i^1+p_i^1 \\
    & - \rho_i^1\mu_i^0-\rho_i^0\mu_i^1)\nabla\theta^0-\frac{1}{\theta^0}(\rho_i^0e_i^0+p_i^0-\rho_i^0\mu_i^0)\nabla\theta^1\Big)=\sum_i(\rho_i^1b_i^0+\rho_i^0b_i^1)-\nabla p^1
\end{split}
\end{equation} where $p^0=\sum_ip_i^0$ and $p^1=\sum_ip_i^1$.

Equation \eqref{cee-eq27} along with \eqref{cee-eq21} and \eqref{cee-eq23} provide a closed system for determining $(\rho_1^0,\dots,\rho_n^0,v^0,\theta^0)$. Now, by \eqref{cee-eq22} follows that $(u_1^1,\dots,u_n^1)$ satisfies the linear system \begin{equation} \label{cee-eq31}
    -\sum_{j\not=i}b_{ij}\theta^0\rho_i^0\rho_j^0(u_i^1-u_j^1)=d_i^0
\end{equation} where $d_i^0=\del_t(\rho_i^0v^0)+\textnormal{div}(\rho_i^0v^0\otimes v^0)-\rho_i^0b_i^0+\rho_i^0\nabla\mu_i^0+\frac{1}{\theta^0}(\rho_i^0e_i^0+p_i^0-\rho_i^0\mu_i^0)\nabla\theta^0$. Since $u_i^0=0$ the first constraint of \eqref{cee-eq11} is satisfied trivially and the second one becomes \begin{equation} \label{cee-constr}
    \sum_i\rho_i^0u_i^1=0.
\end{equation} Moreover, due to \eqref{cee-eq27} and \eqref{cee-eq29} we see that $\sum_id_i^0=0$. This guarantees the solvability of system \eqref{cee-eq31}, since assumptions (i)-(iv) are satisfied according to the analysis of system \eqref{cee-eq19}. Hence, there exists a unique solution $(u_1^1,\dots,u_n^1)$ to system \eqref{cee-eq31}.

Similarly, summing \eqref{cee-eq25} over $i$, using the symmetry of $b_{ij}$, the identity \eqref{cee-eq30} and the constraint \eqref{cee-constr} we get \begin{equation} \label{cee-eq32}
    \del_t(\sum_i\rho_i^1v^0+\rho^0v^1)+\textnormal{div}(\sum_i\rho_i^1v^0\otimes v^0+\rho^0v^1\otimes v^0+\rho^0v^0\otimes v^1)=\sum_i(\rho_i^1b_i^0+\rho_i^0b_i^1)-\nabla p^1
\end{equation} which along with \eqref{cee-eq24} and \eqref{cee-eq26} provide a closed system for determining $(\{\rho_i^1\}_{i=1}^n,v^1,\theta^1)$. Note that after using the constraint \eqref{cee-constr}, \eqref{cee-eq26} reads \[ \begin{split}
    \del_t\Big(\sum_i & ( \rho_i^0e_i^1+\rho_i^1e_i^0)+\sum_i\frac{1}{2}\rho_i^1(v^0)^2+\rho^0v^0v^1\Big)+\textnormal{div}\Big(\Big(\sum_i\rho_i^0e_i^0+\sum_i\frac{1}{2}\rho_i^0(v^0)^2\Big)v^1 \\
    & + \Big(\sum_i(\rho_i^0e_i^1+\rho_i^1e_i^0)+\rho^0v^0v^1+\sum_i\frac{1}{2}\rho_i^1(v^0)^2\Big)v^0\Big)=-\textnormal{div}(p^0v^1+p^1v^0) \\
    & + \sum_i(\rho_i^0b_i^0\cdot v^1+(\rho_i^1b_i^0+\rho_i^0b_i^1)\cdot v^0)+\sum_i\rho_i^0b_i^0\cdot u_i^1+\sum_i(\rho_i^0r_i^1+\rho_i^1r_i^0) \\
    & + \textnormal{div}\left(\kappa^0\nabla\theta^1+\kappa^1\nabla\theta^0-\sum_i(\rho_i^0e_i^0+p_i^0)u_i^1\right)
\end{split} \] 

\subsection{Reconstruction of the effective equations}\label{sec2.2}

Next, we reconstruct the effective equations that are valid asymptotically up to order $\mathcal{O}(\epsilon^2)$. We add back \eqref{cee-eq21} plus $\epsilon$ times \eqref{cee-eq24}, \eqref{cee-eq27} plus $\epsilon$ times \eqref{cee-eq32} and \eqref{cee-eq23} plus $\epsilon$ times \eqref{cee-eq26} to obtain \begin{equation} \label{cee-eq33}
    \del_t(\rho_i^0+\epsilon\rho_i^1)+\textnormal{div}(\rho_i^0v^0+\epsilon(\rho_i^0v^1+\rho_i^1v^0))=-\epsilon\textnormal{div}(\rho_i^0u_i^1)
\end{equation} \begin{equation} \label{cee-eq34}
    \begin{split}
        \del_t\Big(\rho^0v^0 & + \epsilon(\rho^1v^0+\rho^0v^1)\Big)+\textnormal{div}\left(\rho^0v^0+\epsilon(\rho^1v^0\otimes v^0+\rho^0v^1\otimes v^0+\rho^0v^0\otimes v^1)\right) \\
        & = \sum_i\{\rho_i^0b_i^0+\epsilon(\rho_i^1b_i^0+\rho_i^0b_i^1)\}-(\nabla p^0+\epsilon\nabla p^1)
    \end{split}
\end{equation} \begin{equation} \label{cee-eq35}
\begin{split}
    &\del_t\Big(\sum_i\rho_i^0e_i^0+\epsilon\sum_i(\rho_i^0e_i^1+\rho_i^1e_i^0)+\sum_i\frac{1}{2}(\rho_i^0+\epsilon\rho_i^1)(v^0)^2)+\epsilon\sum_i\frac{1}{2}\rho_i^02v^0v^1\Big) \\
    & + \textnormal{div}\Big(\sum_i\rho_i^0e_i^0(v^0+\epsilon v^1)+\epsilon\sum_i(\rho_i^1e_i^0+\rho_i^0e_i^1)v^0+\sum_i\frac{1}{2}\rho_i^0(v^0)^2(v^0+\epsilon v^1) \\
    & + \epsilon\sum_i\frac{1}{2}\rho_i^1(v^0)^2v^0)+\epsilon\sum_i\frac{1}{2}\rho_i^02v^0(u_i^1+v^1)v^0\Big)=-\textnormal{div}(p^0v^0+\epsilon(p^1v^0+p^0v^1) \\
    & + \sum_i(\rho_i^0r_i^0+\epsilon(\rho_i^0r_i^1+\rho_i^1r_i^0))+\sum_i\Big(\rho_i^0b_i^0\cdot v^0+\epsilon(\rho_i^0b_i^0\cdot v^1+(\rho_i^0b_i^1+\rho_i^1b_i^0)\cdot v^0)\Big) \\
    & + \epsilon\sum_i\rho_i^0b_i^0\cdot u_i^1+\textnormal{div}\left((\kappa^0\nabla\theta^0)+\epsilon(\kappa^1\nabla\theta^0+\kappa^1\nabla\theta^0)-\epsilon\sum_i(\rho_i^0e_i^0+p_i^0)u_i^1\right)
\end{split}
\end{equation} 

Now set
\begin{equation} \label{cee-eq36}
    \rho_i^\epsilon=\rho_i^0+\epsilon\rho_i^1+\mathcal{O}(\epsilon^2) \, , \qquad 
    \rho^\epsilon=\sum_i\rho_i^\epsilon
\end{equation} \begin{equation} \label{cee-eq37}
    u_i^\epsilon=u_i^0+\epsilon u_i^1+\mathcal{O}(\epsilon^2)
\end{equation} \begin{equation} \label{cee-eq39}
    v^\epsilon=v^0+\epsilon v^1+\mathcal{O}(\epsilon^2)
\end{equation}\begin{equation} \label{cee-eq40}
    \theta^\epsilon=\theta^0+\epsilon \theta^1+\mathcal{O}(\epsilon^2)
\end{equation} 

We also have the expansions obtained via \eqref{expansione} and reading
\begin{equation}\label{expansioneta2}
\begin{aligned}
e_i^\epsilon &= e_i^0 + \epsilon e_i^1 + \epsilon^2 e_i^2 + \mathcal{O}(\epsilon^3) 
\\
p_i^\epsilon &=  p_i^0 + \epsilon p_i^1 + \epsilon^2 p_i^2 + \mathcal{O}(\epsilon^3) 
\\
\kappa^\epsilon &= \kappa^0 + \epsilon \kappa^1 + \epsilon^2 \kappa^2 + \mathcal{O}(\epsilon^3) 
\end{aligned}
\end{equation}

Then equations \eqref{cee-eq33}-\eqref{cee-eq35} read \begin{equation} \label{cee-eq41}
    \del_t\rho_i^\epsilon+\textnormal{div}(\rho_i^\epsilon v^\epsilon)=-\textnormal{div}(\rho_i^\epsilon u_i^\epsilon)+\mathcal{O}(\epsilon^2)
\end{equation} \begin{equation} \label{cee-eq42}
    \del_t(\rho^\epsilon v^\epsilon)+\textnormal{div}(\rho^\epsilon v^\epsilon\otimes v^\epsilon)=\rho^\epsilon b^\epsilon-\nabla p^\epsilon+\mathcal{O}(\epsilon^2)
\end{equation} \begin{equation} \label{cee-eq43}
\begin{split}
    \del_t\left(\rho^\epsilon e^\epsilon+\frac{1}{2}\rho^\epsilon (v^\epsilon)^2\right) & + \textnormal{div}\left(\left(\rho^\epsilon e^\epsilon+\frac{1}{2}\rho^\epsilon (v^\epsilon)^2\right)v^\epsilon\right)=-\textnormal{div}(p^\epsilon v^\epsilon)+\rho^\epsilon r^\epsilon+\rho^\epsilon b^\epsilon\cdot v^\epsilon \\
    & + \sum_i\rho_i^\epsilon b_i^\epsilon\cdot u_i^\epsilon+\textnormal{div}\left(\kappa^\epsilon\nabla\theta^\epsilon-\sum_i(\rho_i^\epsilon e_i^\epsilon+p_i^\epsilon)u_i^\epsilon\right)+\mathcal{O}(\epsilon^2)
\end{split}
\end{equation} 

Finally, we need to construct the formulas determining $(u_i^\epsilon)_i$. Using \eqref{cee-eq31} we deduce that \begin{equation} \label{cee-eq44}
\begin{split}
    -\sum_{j\not=i}b_{ij}\theta^\epsilon\rho_i^\epsilon\rho_j^\epsilon(u_i^\epsilon-u_j^\epsilon) & = -\epsilon\sum_{j\not=i}b_{ij}\theta^0\rho_i^0\rho_j^0(u_i^1-u_j^1)+\mathcal{O}(\epsilon^2) \\
    & = \epsilon d_i^0+\mathcal{O}(\epsilon^2)
\end{split}
\end{equation} where from \eqref{cee-eq21}, summing \eqref{cee-eq21} over $i$ and \eqref{cee-eq27}: \begin{equation} \label{cee-eq45}
    \begin{split}
        d_i^0 & = \del_t(\rho_i^0v^0)+\textnormal{div}(\rho_i^0v^0\otimes v^0)-\rho_i^0b_i^0+\rho_i^0\nabla\mu_i^0+\frac{1}{\theta^0}(\rho_i^0e_i^0+p_i^0-\rho_i^0\mu_i^0)\nabla\theta^0 \\
        & = \rho_i^0\left(\del_tv^0+v^0\cdot\nabla v^0\right)-\rho_i^0b_i^0+\rho_i^0\nabla\mu_i^0+\frac{1}{\theta^0}(\rho_i^0e_i^0+p_i^0-\rho_i^0\mu_i^0)\nabla\theta^0 \\
        & = \frac{\rho_i^0}{\rho^0}\left(\del_t(\rho^0v^0)+\textnormal{div}(\rho^0v^0\otimes v^0)\right)-\rho_i^0b_i^0+\rho_i^0\nabla\mu_i^0+\frac{1}{\theta^0}(\rho_i^0e_i^0+p_i^0-\rho_i^0\mu_i^0)\nabla\theta^0 \\
        & = \frac{\rho_i^0}{\rho^0}\sum_j(\rho_j^0b_j^0-\nabla p_j^0)-\rho_i^0b_i^0+\rho_i^0\nabla\mu_i^0+\frac{1}{\theta^0}(\rho_i^0e_i^0+p_i^0-\rho_i^0\mu_i^0)\nabla\theta^0
    \end{split}
\end{equation} This motivates to define \begin{equation} \label{cee-eq46}
    d_i^\epsilon:=\frac{\rho_i^\epsilon}{\rho^\epsilon}\sum_j(\rho_j^\epsilon b_j^\epsilon-\nabla p_j^\epsilon)-\rho_i^\epsilon b_i^\epsilon+\rho_i^\epsilon\nabla\mu_i^\epsilon+\frac{1}{\theta^\epsilon}(\rho_i^\epsilon e_i^\epsilon+p_i^\epsilon-\rho_i^\epsilon\mu_i^\epsilon)\nabla\theta^\epsilon
\end{equation} which by \eqref{cee-eq28}-\eqref{cee-eq30} sums up to zero and thus the linear system for determining $\{u_i^\epsilon\}_{i=1}^n$ is \begin{equation} \label{cee-eq47}
    -\sum_{j\not=i}b_{ij}\theta^\epsilon\rho_i^\epsilon\rho_j^\epsilon(u_i^\epsilon-u_j^\epsilon)=\epsilon d_i^\epsilon+\mathcal{O}(\epsilon^2)
\end{equation} and is solvable according to the analysis of \eqref{cee-eq19}. Moreover, the constraint becomes \begin{equation} \label{cee-eq48}
    \sum_i\rho_i^\epsilon u_i^\epsilon=\mathcal{O}(\epsilon^2)
\end{equation} which means that since the variables $\rho_i$ and $u_i$ are defined up to order $\mathcal{O}(\epsilon^2)$ we can set $\sum_i\rho_i^\epsilon u_i^\epsilon=0$ up to that order.

Therefore, omitting the $\epsilon$-notation and the higher-order terms, we conclude that system \eqref{intro-eq1}-\eqref{intro-gd} is approximated within $\mathcal{O}(\epsilon^2)$ by the system: \begin{equation} \label{cee-eq49}
    \del_t\rho_i+\textnormal{div}(\rho_iv)=-\textnormal{div}(\rho_iu_i)
\end{equation} \begin{equation} \label{cee-eq50}
    \del_t(\rho v)+\textnormal{div}(\rho v\otimes v)=\rho b-\nabla p
\end{equation}
 \begin{equation} \label{cee-eq51}
\begin{split}
    \del_t\left(\rho e+\frac{1}{2}\rho v^2\right)+\textnormal{div}\left((\rho e+\frac{1}{2}\rho v^2)v\right) & = -\textnormal{div}(pv)+\rho r+\rho b\cdot v+\sum_i\rho_ib_i\cdot u_i \\
    & + \textnormal{div}\left(\kappa\nabla\theta-\sum_i(\rho_ie_i+p_i)u_i\right)
\end{split}
\end{equation}  where $u_i$ are determined by solving the constrained linear system: 
\begin{equation} \label{cee-eq52}
\begin{aligned}
    -\sum_{j\not=i}b_{ij}\theta\rho_i\rho_j(u_i-u_j)&=\epsilon\left(\frac{\rho_i}{\rho}(-\nabla p+\rho b)+\rho_i\theta\nabla\frac{\mu_i}{\theta}-\theta(\rho_ie_i+p_i)\nabla\frac{1}{\theta}-\rho_ib_i\right)
\\
    \sum_i\rho_iu_i &=0
\end{aligned} 
\end{equation}
Our system \eqref{cee-eq49}-\eqref{cee-eq52} agrees with the Type-I model obtained in \cite{BDre} using the entropy invariant method, in the case of an inviscid, non-reactive mixture with zero thermal diffusivities \cite[Section 8]{BDre}. We note here that the present work concerns only the derivation of the corresponding Type-I model in the formal level. For a validation of the expansion and a proof of the convergence, however valid  in the isothermal case, we refer to \cite{BGP,HJT}.

\subsection{Asymptotic expansion of the entropy equation}\label{sec2.3}

As explained in  Appendix \ref{appC} or in \cite[Section 5]{BDre}, given the Type-II system \eqref{intro-eq1}-\eqref{intro-eq3}, one derives the entropy equation: 
\begin{equation}\label{eq-entropytypeII}
\begin{split}
    \del_t(\rho\eta)+\textnormal{div}(\rho\eta v) & = \textnormal{div}\left(\frac{\kappa\nabla\theta-\sum_i(\rho_ie_i+p_i-\rho_i\mu_i)u_i}{\theta}\right)+\frac{1}{\theta^2}\kappa|\nabla\theta|^2 \\
    & + \frac{1}{2\epsilon}\sum_i\sum_jb_{ij}\rho_i\rho_j(u_i-u_j)^2+\frac{\rho r}{\theta}
\end{split}
\end{equation} There are two approaches for deriving the entropy equation of the asymptotic limit system  \eqref{cee-eq49}-\eqref{cee-eq51}: (i) to derive it by using the system and the constitutive relations as was done in Appendix \ref{appC} to obtain \eqref{eq-entropytypeII} and  (ii) to expand the entropy equation of the Type-II model and obtain its $\epsilon^2-$approximation the same way we obtained system \eqref{cee-eq49}-\eqref{cee-eq51}. Here we present the latter way, but it is easy to verify that the two results coincide.

As in subsection \ref{sec2.1}, we introduce the same Hilbert expansion, insert it into \eqref{eq-entropytypeII} and identify terms of the same order:

(i) Terms of order $\mathcal{O}(1/\epsilon)$: \begin{equation} \label{cee-eq55} \frac{1}{2}\sum_{i,j}b_{ij}\rho_i^0\rho_j^0(u_i^0-u_j^0)^2=0 \end{equation} 
which is consistent with  $u_i^0=0$ in subsection \ref{sec2.1}.

Next, using  $u_i^0=0$,  the expansion \eqref{eq-exp1theta} for  $\frac{1}{\theta}$  and 
\[ \frac{1}{\theta^2}=\frac{1}{(\theta^0)^2}-\epsilon\frac{2\theta^1}{(\theta^0)^3}+\epsilon^2\frac{3(\theta^1)^2-\theta^2\theta^0}{(\theta^0)^4}+\mathcal{O}(\epsilon^3) \, , 
\] 
as well as the expansions \eqref{expansioneta}, \eqref{expansioneta2}
we obtain:

(ii) at order $\mathcal{O}(1)$: 
\begin{equation} \label{cee-eq58}
    \begin{split}
    \del_t(\rho^0\eta^0) & + \textnormal{div}(\rho^0\eta^0v^0)=\textnormal{div}\left(\frac{1}{\theta^0}\kappa^0\nabla\theta^0\right)+\frac{1}{(\theta^0)^2}\kappa^0(\nabla\theta^0)^2+\frac{\rho^0r^0}{\theta^0}
    \end{split}
\end{equation} 
(iii) at order $\mathcal{O}(\epsilon)$: 
\begin{equation} \label{cee-eq59}
    \begin{split}
        & \del_t(\rho^1\eta^0+\rho^0\eta^1)+\textnormal{div}((\rho^1\eta^0+\rho^0\eta^1)v^0+\rho^0\eta^0v^1)=\textnormal{div}\Big(\frac{1}{\theta^0}(\kappa^1\nabla\theta^0+\kappa^0\nabla\theta^1) \\
        & -\frac{\theta^1}{(\theta^0)^2}\kappa^0\nabla\theta^0-\frac{1}{\theta^0}\sum_i(\rho_i^0e_i^0+p_i^0-\rho_i^0\mu_i^0)u_i^1\Big)-\frac{2\theta^1}{(\theta^0)^3}\kappa^0(\nabla\theta^0)^2+\frac{1}{(\theta^0)^2}\kappa^1(\nabla\theta^0)^2 \\
        & + \frac{1}{(\theta^0)^2}\kappa^02\nabla\theta^0\nabla\theta^1+\frac{1}{2}\sum_{i,j}b_{ij}\rho_i^0\rho_j^0(u_i^1-u_j^1)^2+\frac{\rho^1r^0+\rho^0r^1}{\theta^0}-\frac{\rho^0r^0\theta^1}{(\theta^0)^2}
    \end{split}
\end{equation} Moreover, using \eqref{cee-eq31} in the entropy expansion, the third last term of \eqref{cee-eq59} reads \[ \begin{split}
    \frac{1}{2}\sum_{i,j}b_{ij}\rho_i^0\rho_j^0(u_i^1-u_j^1)^2 
    & = \frac{1}{2}\sum_iu_i^1\sum_jb_{ij}\rho_i^0\rho_j^0(u_i^1-u_j^1)-\frac{1}{2}\sum_ju_j^1\sum_ib_{ij}\rho_i^0\rho_j^0(u_i^1-u_j^1) \\
    & = -\frac{1}{2\theta^0}\sum_iu_i^1d_i^0-\frac{1}{2\theta^0}\sum_ju_j^1d_j^0 \\
    & = -\frac{1}{\theta^0}\sum_iu_i^1\cdot d_i^0
\end{split} \] Thus, in the reconstruction, we add \eqref{cee-eq58} plus $\epsilon$ times \eqref{cee-eq59}, to obtain \begin{equation} \label{cee-eq60}
    \begin{split}
        & \del_t(\rho^0\eta^0+\epsilon(\rho^1\eta^0+\rho^0\eta^1))+\textnormal{div}(\rho^0\eta^0v^0+\epsilon((\rho^1\eta^0+\rho^0\eta^1)v^0+\rho^0\eta^0v^1)) \\
        & = \textnormal{div}\left[\frac{1}{\theta^0}\kappa^0\nabla\theta^0+\epsilon\left(\frac{1}{\theta^0}(\kappa^1\nabla\theta^0+\kappa^0\nabla\theta^1)-\frac{\theta^1}{(\theta^0)^2}\kappa^0\nabla\theta^0-\frac{1}{\theta^0}\sum_i(\rho_i^0e_i^0+p_i^0-\rho_i^0\mu_i^0)u_i^1\right)\right] \\
        & + \frac{1}{(\theta^0)^2}\kappa^0(\nabla\theta^0)^2+\epsilon\left(-\frac{2\theta^1}{(\theta^0)^3}\kappa^0(\nabla\theta^0)^2+\frac{1}{(\theta^0)^2}\kappa^1(\nabla\theta^0)^2+\frac{1}{(\theta^0)^2}\kappa^02\nabla\theta^0\nabla\theta^1\right) \\
        & -\frac{\epsilon}{\theta^0}\sum_iu_i^1\cdot d_i^0+\frac{\rho^0r^0}{\theta^0}+\epsilon\left(\frac{\rho^1r^0+\rho^0r^1}{\theta^0}-\frac{\rho^0r^0\theta^1}{(\theta^0)^2}\right)
    \end{split}
\end{equation} which in turn gives \begin{equation} \label{cee-eq61}
    \begin{split}
        \del_t(\rho^\epsilon\eta^\epsilon)+\textnormal{div}(\rho^\epsilon\eta^\epsilon v^\epsilon) & = \textnormal{div}\left(\kappa^\epsilon\frac{\nabla\theta^\epsilon}{\theta^\epsilon}-\sum_i\frac{\rho_i^\epsilon e_i^\epsilon+p_i^\epsilon-\rho_i^\epsilon\mu_i^\epsilon}{\theta^\epsilon}u_i^\epsilon\right)+\kappa^\epsilon\frac{(\nabla\theta^\epsilon)^2}{(\theta^\epsilon)^2} \\
        & +\frac{\rho^\epsilon r^\epsilon}{\theta^\epsilon}-\frac{1}{\theta^\epsilon}\sum_iu_i^\epsilon\cdot d_i^\epsilon+\mathcal{O}(\epsilon^2)
    \end{split}
\end{equation} 
Omitting the $O(\epsilon^2)$ correction, the entropy equation for the Type-I system \eqref{cee-eq49}-\eqref{cee-eq51} reads 
\begin{equation}\label{eq-entropytypeI}
    \begin{split}
        \del_t(\rho\eta)+\textnormal{div}(\rho\eta v) & = \textnormal{div}\left(\frac{1}{\theta}\kappa\nabla\theta-\frac{1}{\theta}\sum_i(\rho_i e_i+p_i-\rho_i\mu_i)u_i\right) \\
        & +\frac{\rho r}{\theta}+\frac{1}{\theta^2}\kappa|\nabla\theta|^2-\frac{1}{\theta}\sum_iu_i\cdot d_i \, .
    \end{split}
\end{equation}
We conclude that the order $O(\epsilon^2)$ expansion of the entropy dissipation structure \eqref{eq-entropytypeII} for the Type-II model yields the entropy 
dissipation structure \eqref{eq-entropytypeI} for the emerging Type-I model.



\section{Dissipative structure of the limiting hyperbolic-parabolic system}\label{sec3}

For simplicity we select the external fields $r_i$ and $b_i$ to be zero and proceed to study the mathematical structure of the Type-I model, which can be written as follows: 
\begin{align} \label{dis-eq1}
    \del_t\rho_i+\textnormal{div}(\rho_iv) &=\textnormal{div}(-\rho_iu_i)
\\
\label{dis-eq2}
    \del_t (\rho v)+\textnormal{div}(\rho v\otimes v+p\mathbb{I}) &=0
\\
\label{dis-eq3}
    \del_t\left(\rho e+\frac{1}{2}\rho v^2\right)+\textnormal{div}\left((\rho e+\frac{1}{2}\rho v^2)v+pv\right)&=\textnormal{div}\left(\kappa\nabla\theta-\sum_i(\rho_ie_i+p_i)u_i\right)
\end{align} 
where $u_i$ are determined by inverting the constrained linear system
\begin{align} 
\label{dis-eq4}
    -\sum_{j\not=i}b_{ij}\theta\rho_i\rho_j(u_i-u_j) &= \epsilon\left(-\frac{\rho_i}{\rho}\nabla p+\rho_i\theta\nabla\frac{\mu_i}{\theta}-\theta(\rho_ie_i+p_i)\nabla\frac{1}{\theta}\right)
\\
\label{dis-eq5}
    \sum_i\rho_iu_i &= 0 \, .
\end{align}
Given the constitutive choices outlined in section \ref{sec2} (see also Appendix \ref{appB}) smooth solutions of \eqref{dis-eq1}-\eqref{dis-eq5} satisfy the entropy identity
\begin{equation} \label{dis-eq6}
    \del_t(-\rho\eta)+\textnormal{div}(-\rho\eta v)=\textnormal{div}\left(\frac{1}{\theta}\sum_i(\rho_ie_i+p_i-\rho_i\mu_i)u_i-\frac{1}{\theta}\kappa\nabla\theta\right)-\frac{1}{\theta^2}\kappa|\nabla\theta|^2+\frac{1}{\theta}\sum_iu_i\cdot d_i
\end{equation}  
Notice we multiplied the entropy identity by minus one, because we want to study the mathematical entropy, which is defined as the negative thermodynamic entropy.

Next, setting  $U=[\rho_1,\dots,\rho_n,v,\theta]^\top$ with $v=(v_1,v_2,v_3)$, system \eqref{dis-eq1}-\eqref{dis-eq3} can be written in the form: \begin{equation} \label{dis-eq7}
    \del_t(A(U))+\textnormal{div}(F(U)) = \textnormal{div}(\varepsilon B(U)\nabla U)
\end{equation} where \begin{equation} \label{dis-eq8}
    A(U)=\begin{bmatrix}
\rho_1 \\
\vdots \\
\rho_n \\
\rho v_1 \\
\rho v_2 \\
\rho v_3 \\
\rho e+\frac{1}{2}\rho v^2
\end{bmatrix}, ~~ F(U)=\begin{bmatrix}
\rho_1v_1 & \rho_1v_2 & \rho_1v_3 \\
\vdots & \vdots & \vdots \\
\rho_nv_1 & \rho_nv_2 & \rho_nv_3 \\
\rho v_1^2+p & \rho v_1v_2 & \rho v_1v_3 \\
\rho v_1v_2 & \rho v_2^2+p & \rho v_2v_3 \\
\rho v_1v_3 & \rho v_2v_3 & \rho v_3^2+p \\
(\rho e+\frac{1}{2}\rho v^2+p)v_1 & (\rho e+\frac{1}{2}\rho v^2+p)v_2 & (\rho e+\frac{1}{2}\rho v^2+p)v_3
\end{bmatrix}
\end{equation}

 \[
    B^\kappa(U)=\begin{bmatrix}
0 & \cdots & 0 & 0 & 0 & 0 & 0  \\
\vdots & \ddots & \vdots & \vdots & \vdots & \vdots & \vdots \\
0 & \cdots & 0 & 0 & 0 & 0 & 0  \\
0 & \cdots & 0 & 0 & 0 & 0 & 0  \\
0 & \cdots & 0 & 0 & 0 & 0 & 0  \\
0 & \cdots & 0 & 0 & 0 & 0 & 0  \\
0 & \cdots & 0 & 0 & 0 & 0 & 1  \\
\end{bmatrix},  ~~ \nabla U=\begin{bmatrix}
\nabla\rho_1 \\
\vdots \\
\nabla\rho_n \\
\nabla v_1 \\
\nabla v_2 \\
\nabla v_3 \\
\nabla\theta
\end{bmatrix} \] 

\[ B^\epsilon(U)=\begin{bmatrix}
A_{11} & \cdots & A_{1n} & 0 & 0 & 0 & A_{1,n+4} \\
\vdots & \ddots & \vdots & 0 & 0 & 0 & \vdots \\
A_{n1} & \cdots & A_{nn} & 0 & 0 & 0 & A_{n,n+4} \\
0 & \cdots & 0 & 0 & 0 & 0 & 0 \\
0 & \cdots & 0 & 0 & 0 & 0 & 0 \\
0 & \cdots & 0 & 0 & 0 & 0 & 0 \\
\sum_{i=1}^n\frac{h_i}{\rho_i}A_{i1} & \cdots & \sum_{i=1}^n\frac{h_i}{\rho_i}A_{in} & 0 & 0 & 0 & \sum_{i=1}^n\frac{h_i}{\rho_i}A_{i,n+4}
\end{bmatrix} \]  where we denote $h_i:=\rho_ie_i+p_i$ and write \[ -\rho_iu_i=\sum_{j=1}^nA_{ij}\nabla\rho_j+A_{i,n+4}\nabla\theta, ~ \textnormal{for $i=1,\dots,n$} \] and in \eqref{dis-eq7} we set $\varepsilon=(\kappa,\epsilon)$ and $B(U)=(B^\kappa(U),B^\epsilon(U))$. We note that the coefficients $A_{ij}$ and $A_{i,n+4}$ are determined by solving system \eqref{dis-eq4}-\eqref{dis-eq5}, as explained later in this section. 

According to \cite{CT} system \eqref{dis-eq7} fits into the class of hyperbolic-parabolic systems (see also \cite{Kaw}) provided the following conditions hold: 

(i) $A$ is a $C^2$ and bijective map from its domain onto its range, with $\nabla A(U)$ non-singular for any $U$ in the domain of $A$

(ii) there is an entropy-entropy flux pair $(H(U), Q_i(U))$, $i=1,2,3$,  generated by a smooth (vector valued) multiplier $G(U)$, such that \[ \nabla H =G\cdot\nabla A  \quad \textnormal{and}  \quad \nabla Q_i =G\cdot\nabla F_i 
\quad \textnormal{for} ~ i=1,2,3 \, .  \] 

(iii) the matrix $\nabla^2 H - G\cdot\nabla^2A$ is symmetric and positive definite 

(iv) smooth solutions of  \eqref{dis-eq7} satisfy an additional conservation law (the so called entropy identity)
\begin{equation} \label{dis-entreq}
\del_t(H(U) )+\text{div}(Q(U) )=\text{div}(G(U)\varepsilon B(U)\nabla U)-\nabla U^\top\nabla G(U)^\top \varepsilon B(U)\nabla U 
\end{equation}
with the matrices $\nabla G^\top B$ inducing entropy dissipation, i.e. \[ \nabla U^\top\nabla G(U)^\top \varepsilon B(U)\nabla U\ge0 \, . \]

Let us check whether these conditions are satisfied in the case of system \eqref{dis-eq1}-\eqref{dis-eq5}. Indeed, setting
\begin{equation}\label{dis-entflux}
H(U) := - \rho \eta(\rho_1 , \dots , \rho_n, \theta) \, , 
\quad
Q_i (U) := - \rho \eta(\rho_1 , \dots , \rho_n , \theta )  v_i  \, , 
\end{equation}
we see that \eqref{dis-entreq} is precisely equation \eqref{dis-eq6}. 
Concerning (i), if $U=[\rho_1,\dots,\rho_n,v,\theta]^\top$ and $\tilde{U}=[\tilde{\rho}_1,\dots,\tilde{\rho}_n,\tilde{v},\tilde{\theta}]^\top$ we have \[ A(U)=A(\tilde{U})\Leftrightarrow \left\{\begin{array}{c}
     \rho_i=\tilde{\rho}_i \\
     v=\tilde{v} \\
     e(\rho_1,\dots,\rho_n,\theta)=e(\tilde{\rho}_1,\dots,\tilde{\rho}_n,\tilde{\theta})
\end{array}\right. \] where the last equality implies that $\theta=\tilde{\theta}$, provided the internal energy is an increasing function of the temperature thus one-to-one. This follows from 
assuming that $e_\theta\equiv c_v>0$, a property connected to the stability of equilibrium states (see \cite{Cal} for more). Hence, if $\psi\in C^3$ and $\rho>0$, $A$ is $C^2$ and bijective in its domain, with \begin{equation} \label{dis-eq10}
\begin{split}
    \textnormal{det}(\nabla A) & =\textnormal{det}\begin{bmatrix}
    1 & 0 & \dots & 0 & 0 & 0 & 0 & 0 \\
    0 & 1 & \dots & 0 & 0 & 0 & 0 & 0 \\
    \vdots & \vdots & \ddots & \vdots & \vdots & \vdots & \vdots & \vdots \\
    0 & 0 & \dots & 1 & 0 & 0 & 0 & 0 \\
    v_1 & v_1 & \dots & v_1 & \rho & 0 & 0 & 0 \\
    v_2 & v_2 & \dots & v_2 & 0 & \rho & 0 & 0 \\
    v_3 & v_3 & \dots & v_3 & 0 & 0 & \rho & 0 \\
    (\rho e)_{\rho_1}+\frac{1}{2}v^2 & (\rho e)_{\rho_2}+\frac{1}{2}v^2 & \dots & (\rho e)_{\rho_n}+\frac{1}{2}v^2 & \rho v_1 & \rho v_2 & \rho v_3 & \rho c_v 
\end{bmatrix} \\
& = \rho^4c_v>0
\end{split}
\end{equation} 

For (ii) using the notation $\rho\hat{\eta}=\rho\eta(\rho_1,\dots,\rho_n,\theta)$ and $\rho\tilde{\eta}=\rho\eta(\rho_1,\dots,\rho_n,\rho e)$ 
and the  thermodynamic relations from Appendix \ref{appA} and \ref{appB}
we compute the partial derivatives of the entropy \[ (\rho\hat{\eta})_{\rho_i}=(\rho\tilde{\eta})_{\rho e}(\rho\hat{e})_{\rho_i}+(\rho\tilde{\eta})_{\rho_i}=\frac{(\rho e)_{\rho_i}-\mu_i}{\theta} \] \[ (\rho\hat{\eta})_\theta=\left(\frac{\rho e-\rho\psi}{\theta}\right)_\theta=\frac{[(\rho e)_\theta-(\rho\psi)_\theta]\theta-(\rho e-\rho\psi)}{\theta^2}= \frac{\rho e_\theta}{\theta}
=\frac{\rho c_v}{\theta}. \] By virtue of the relation $\nabla(-\rho\eta)=G\cdot\nabla A$ this determines the multiplier $G$ to be 
\begin{equation} \label{dis-eq11}
    G(U)= \bigg ( G_j (U) \bigg )_{j =1}^{n+4} = \frac{1}{\theta}\begin{bmatrix}
 \mu_1-\frac{1}{2}v^2 \\
 \vdots \\
 \mu_n-\frac{1}{2}v^2 \\
 v_1 \\
 v_2 \\
 v_3 \\
 -1
\end{bmatrix}
\end{equation} while relations $\nabla(-\rho\eta v_i)=G\cdot\nabla F_i$ serve as a way of verifying that our calculations are correct; indeed, they hold for the $G$ we found, 
using the properties
\begin{equation} \label{dis-eq12}
    (\mu_i)_\theta=((\rho\psi)_{\rho_i})_\theta=((\rho\psi)_\theta)_{\rho_i}=(-\rho\hat{\eta})_{\rho_i}=\frac{\mu_i-(\rho e)_{\rho_i}}{\theta}
\end{equation}

Now, for (iii), note that $\displaystyle{ G\cdot\nabla^2A:=\sum_{j =1}^{n+4} G_j (U) \nabla^2A_j(U)}$.  One easily sees that $\nabla^2 A_j = 0$ for $j =1 , \dots , n$. Then a tedious
but straightforward calculation, using the thermodynamic relations  \eqref{intro-const1}-\eqref{intro-const4}  and
$$
\begin{aligned}
\frac{1}{\theta} (\rho e)_\theta  -  (\rho \eta)_\theta &= \frac{1}{\theta}  \big (  (\rho \psi )_\theta  + \rho \eta \big ) = 0
\\
\frac{1}{\theta} (\rho e)_{\theta \theta} -  (\rho \eta)_{\theta \theta} &= \frac{1}{\theta} (\rho \psi )_{\theta \theta}  + \frac{2}{\theta}  (\rho \eta)_\theta
= \frac{\rho}{\theta}  \eta_\theta  = \frac{\rho}{\theta^2} e_\theta \, ,
\end{aligned}
$$
implies
 \begin{equation} 
\label{dis-eq13}
    \nabla^2(-\rho\eta)-G\cdot\nabla^2A=\frac{1}{\theta}\begin{bmatrix}
    (\rho\psi)_{\rho_1\rho_1} & \cdots & (\rho\psi)_{\rho_1\rho_n} & 0 & 0 & 0 & 0 \\
    \vdots & \ddots & \vdots & \vdots & \vdots & \vdots & \vdots \\
    (\rho\psi)_{\rho_n\rho_1} & \cdots & (\rho\psi)_{\rho_n\rho_n} & 0 & 0 & 0 & 0 \\
    0 & \cdots & 0 & \rho & 0 & 0 & 0 \\
    0 & \cdots & 0 & 0 & \rho & 0 & 0 \\
    0 & \cdots & 0 & 0 & 0 & \rho & 0 \\
    0 & \cdots & 0 & 0 & 0 & 0 & \frac{1}{\theta}\rho c_v \\
\end{bmatrix}
\end{equation} 
Clearly,  for  $\psi\in C^3$,  $\rho, \theta >0$ and $c_v > 0$, the latter will be positive definite provided the matrix $(\rho\psi)_{\rho_i\rho_j}$ is positive definite. Indeed, for any $\xi\in\mathbb{R}^{n+4}\setminus\{(0,\dots,0)\}$ we see that \[ \xi^\top(\nabla^2(-\rho\eta)-G\cdot\nabla^2A)\xi=\frac{1}{\theta}\sum_{i=1}^n\sum_{j=1}^n(\rho\psi)_{\rho_i\rho_j}\xi_i\xi_j+\frac{\rho}{\theta}(\xi_{n+1}^2+\xi_{n+2}^2+\xi_{n+3}^2)+\frac{\rho c_v}{\theta^2}\xi_{n+4}^2>0 \] The assumption that $(\rho\psi)_{\rho_i\rho_j}$ is positive definite is a natural assumption in thermodynamics that is related to the convexity of the entropy (again we refer to \cite{Cal} regarding the stability of equilibrium states, or to \cite{BDre}).

For the last condition, we need to show that $\nabla U^\top\nabla G(U)^\top \varepsilon B(U)\nabla U \ge 0$, in other words, the entropy production is non-negative: \begin{equation} \label{dis-eq14}
    \zeta=\frac{1}{\theta^2}\kappa|\nabla\theta|^2-\frac{1}{\theta}\sum_iu_i\cdot d_i\ge0.
\end{equation} Clearly the first term is non-negative, since $\kappa\ge0$. For the second term we need to invert system \eqref{dis-eq4}, in order to determine $\rho_iu_i$. 

Various methods for the inversion of the Maxwell-Stefan system are available in the literature, see for example \cite{BDre,HJT} and \cite[Sec. 7.7]{GioMulti}. In the present work, we invert \eqref{dis-eq4} using the Bott-Duffin inverse, following the analysis of \cite{MaxS}. The advantage of this method is that it provides an explicit formula for the solution of the linear system \eqref{dis-eq4},  \eqref{dis-eq5}, which is helpful to estimate the dissipation.

The need to introduce the Bott-Duffin inverse arises from the fact that the desired inversion has to respect the constraint $\sum_i\rho_iu_i=0$, i.e. we invert within the appropriate subspace. More precisely, we consider the solution of the generic system \begin{equation} \label{dis-eq15}
    Mx+y=w, ~~ x\in L,~y\in L^\perp
\end{equation} where $M\in\mathbb{R}^{m\times m}$, for some $m\in\mathbb{N}$, is a matrix and $L\subset\mathbb{R}^m$ a subspace. Let $\mathbb{P}_L$ and $\mathbb{P}_{L^\perp}$ be the projection operators onto the subspaces $L$ and $L^\perp$, respectively. Then, the set of solutions of system \eqref{dis-eq15} is the same as the set of solutions of the system \begin{equation} \label{dis-eq16}
    (M\mathbb{P}_L+\mathbb{P}_{L^\perp})z=w
\end{equation} and $[x,y]^\top$ solves \eqref{dis-eq15} if and only if $x=\mathbb{P}_Lz$ and $y=\mathbb{P}_{L^\perp}z=w-M\mathbb{P}_Lz$. Now, if the matrix $M\mathbb{P}_L+\mathbb{P}_{L^\perp}$ is invertible, we define the Bott-Duffin inverse of $M$ with respect to $L$ by \begin{equation} \label{dis-eq17}
    M^{BD}=\mathbb{P}_L(M\mathbb{P}_L+\mathbb{P}_{L^\perp})^{-1}
\end{equation} so that the solution of \eqref{dis-eq15} is given by \begin{equation} \label{dis-eq18}
    x=M^{BD}d, ~~ y=w-Mx
\end{equation} In our context, if we introduce the molar fractions \begin{equation} \label{dis-eq19}
    c_i=\frac{\rho_i}{\rho}
\end{equation} the left-hand side of \eqref{dis-eq4} reads: \[ \begin{split}
    -\sum_{j\not=i}b_{ij}\theta\rho_i\rho_j(u_i-u_j) & = -\rho\theta\left(\sqrt{\rho_i}\sum_{j\not=i}b_{ij}c_j(\sqrt{\rho_i}u_i)-\sqrt{\rho_i}\sum_{j\not=i}b_{ij}\sqrt{c_i}\sqrt{c_j}(\sqrt{\rho_j}u_j)\right) \\
    & = -\rho\theta\sqrt{\rho_i}\sum_{j=1}^n\left(\sum_{k\not=i}c_kb_{ik}\delta_{ij}-\sqrt{c_ic_j}b_{ij}\right)\sqrt{\rho_j}u_j \\
    & = -\rho\theta\sqrt{\rho_i}\sum_{j=1}^nM_{ij}\sqrt{\rho_j}u_j
\end{split} \] where we introduce the matrix $M=(M_{ij})$ given by \begin{equation} \label{dis-eq20}
    M_{ij}=\left\{\begin{array}{cc}
        \sum_{k\not=i}c_kb_{ik} & i=j \\
        -\sqrt{c_ic_j}b_{ij} & i\not=j 
    \end{array}\right.
\end{equation} and we are interested in the constrained inversion $Mx=w$, $x\in L$, where $w_i=-\frac{\epsilon d_i}{\rho\theta\sqrt{\rho_i}}$, using $d_i$ from section \ref{sec2}, $x_i=\sqrt{\rho_i}u_i$ and $L=\{(y_1,\dots,y_n)\in\mathbb{R}^n:\sum_{i=1}^n\sqrt{\rho_i}y_i=0\}$. Moreover, the projection matrix $\mathbb{P}_L$ on $L$ is given by \begin{equation} \label{dis-eq21}
    (\mathbb{P}_L)_{ij}=\delta_{ij}-\frac{\sqrt{\rho_i\rho_j}}{\rho}
\end{equation} 

In \cite{MaxS} it was proven that the matrix $M$ from \eqref{dis-eq20} satisfies the relation \begin{equation} \label{dis-eq22}
    z^\top Mz\ge\mu|\mathbb{P}_Lz|^2~~\forall z\in\mathbb{R}^n ~~ \textnormal{where $\mu=\min_{i\not=j}b_{ij}$}
\end{equation} which in turn implies that the Bott-Duffin inverse of $M$, namely $M^{BD}$, is well-defined, symmetric and satisfies \begin{equation} \label{dis-eq23}
    z^\top M^{BD}z\ge\lambda|\mathbb{P}_Lz|^2~~\forall z\in\mathbb{R}^n ~~ \textnormal{where $\lambda=\left(2\sum_{i\not=j}(b_{ij}+1)\right)^{-1}$}
\end{equation} Therefore, system \eqref{dis-eq4} is written \begin{equation} \label{dis-eq24}
    \sum_{j=1}^nM_{ij}\sqrt{\rho_j}u_j=-\frac{\epsilon d_i}{\rho\theta\sqrt{\rho_i}}
\end{equation} and can be inverted  \begin{equation} \label{dis-eq25}
    \sqrt{\rho_i}u_i=-\sum_{j=1}^nM_{ij}^{BD}\frac{\epsilon d_j}{\rho\theta\sqrt{\rho_j}}
\end{equation} Plugging \eqref{dis-eq25} into the last term of \eqref{dis-eq14} we get \[ \begin{split}
    -\frac{1}{\theta}\sum_iu_i\cdot d_i & = -\frac{1}{\theta}\sum_i\sqrt{\rho_i}u_i\cdot \frac{d_i}{\sqrt{\rho_i}} \\
    & = \frac{\epsilon}{\rho\theta^2}\sum_i\sum_jM_{ij}^{BD}\frac{d_i}{\sqrt{\rho_i}}\frac{d_j}{\sqrt{\rho_j}} \\
    & \ge 0
\end{split} \] by \eqref{dis-eq23}. Hence, \eqref{dis-eq14} is satisfied, which means that the matrices $\nabla G^\top B$ induce entropy dissipation as we wanted, concluding that system \eqref{dis-eq1}-\eqref{dis-eq3} along with the linear system \eqref{dis-eq4} subject to the constraint \eqref{dis-eq5} is of hyperbolic-parabolic type. In fact, condition \eqref{dis-eq14} is the minimum framework inducing entropy dissipation along the evolution and at the same time allowing for degenerate diffusion matrices (see \cite{CT}). In section \ref{sec5} we shall show that our problem enjoys a stronger dissipative structure, which allows us to establish some convergence results.

A particular case of system \eqref{dis-eq1}-\eqref{dis-eq5} is the system without mass-diffusion and heat-conduction, obtained by setting $\kappa=\epsilon=0$: \begin{equation} \label{dis-eq27}
    \del_t\rho_i+\textnormal{div}(\rho_iv)=0
\end{equation} \begin{equation} \label{dis-eq28}
    \del_t(\rho v)+\textnormal{div}(\rho v\otimes v+p\mathbb{I})=0
\end{equation} \begin{equation} \label{dis-eq29}
    \del_t\left(\rho e+\frac{1}{2}\rho v^2\right)+\textnormal{div}\left((\rho e+\frac{1}{2}\rho v^2)v+pv\right)=0
\end{equation} equipped with the entropy identity \begin{equation} \label{dis-eq30}
    \del_t(-\rho\eta)+\textnormal{div}(-\rho\eta v)=0
\end{equation} We would like to show that \eqref{dis-eq27}-\eqref{dis-eq29} is hyperbolic under the assumptions $(\rho\psi)_{\rho_i\rho_j}$ is positive definite and $e_\theta>0$, as long as the total mass $\rho$ remains away from zero. To do so, one needs to rewrite system \eqref{dis-eq27}-\eqref{dis-eq29} in the form \[ \nabla A(U)\del_tU+\sum_{\alpha=1}^3\nabla F_\alpha(U)\del_{x_\alpha}U=0 \] where $A,F_\alpha,U$ are as in \eqref{dis-eq7}. Since $\nabla A$ is non-singular, we proceed to find the characteristic speeds of the system by solving, for any $N=(N_1,N_2,N_3)$ on the sphere, the eigenvalue problem \[ \left[\sum_{\alpha=1}^3\nabla F_\alpha(U) N_\alpha-\lambda(U,N)\nabla A(U)\right]r(U,N)=0 \] where $\lambda,r$ are the eigenvalues and eigenvectors respectively and \[ \begin{split} &\sum_{\alpha=1}^3\nabla F_\alpha(U) N_\alpha-\lambda(U,N)\nabla A(U) \\
& =\begin{bmatrix} v\cdot N-\lambda & \cdots & 0 & \rho_1N_1 & \rho_1N_2 & \rho_1N_3 & 0 \\
\vdots & \ddots & \vdots & \vdots & \vdots & \vdots & \vdots \\
0 & \cdots & v\cdot N-\lambda & \rho_nN_1 & \rho_nN_2 & \rho_nN_3 & 0 \\
p_{\rho_1}N_1 & \cdots & p_{\rho_n}N_1 & \rho(v\cdot N-\lambda) & 0 & 0 & N_1p_\theta \\
p_{\rho_1}N_2 & \cdots & p_{\rho_n}N_2 & 0 & \rho(v\cdot N-\lambda) & 0 & N_2p_\theta \\
p_{\rho_1}N_3 & \cdots & p_{\rho_n}N_3 & 0 & 0 & \rho(v\cdot N-\lambda) & N_3p_\theta \\
0 & \cdots & 0 & N_1\theta p_\theta & N_2\theta p_\theta & N_3\theta p_\theta & \rho c_v(v\cdot N-\lambda)
\end{bmatrix}
\end{split} \] 
By using the property \[ \textnormal{det}\begin{bmatrix} A & B \\
C & D \end{bmatrix}=\textnormal{det}(A)\textnormal{det}(D-CA^{-1}B) \] from \cite[Sec. 5]{JSi}, for any block of matrices $A,B,C,D$, we determine the characteristic equation: \[ \rho^4c_v(v\cdot N-\lambda)^{n+2}\left((v\cdot N-\lambda)^2-\frac{1}{\rho}\sum_i\rho_ip_{\rho_i}-\frac{\theta p_\theta^2}{c_v\rho^2}\right)=0 \] which yields the wave speeds \[ \lambda_1=\cdots=\lambda_{n+2}=v\cdot N,~\lambda_{n+3,n+4}=v\cdot N\pm\sqrt{\frac{1}{\rho}\sum_i\rho_ip_{\rho_i}+\frac{\theta p_\theta^2}{c_v\rho^2}} \] Therefore, system \eqref{dis-eq27}-\eqref{dis-eq29} is hyperbolic if all eigenvalues are real, which holds under the hypotheses $\rho,\theta>0$, $c_v>0$ and $(\rho\psi)_{i,j}>0$. The last hypothesis ensures that the term $\sum_i\rho_ip_{\rho_i}$ is positive, since by \eqref{appC-eq3} \[ \sum_i\rho_ip_{\rho_i}=\sum_{i,j}\rho_i\rho_j(\mu_i)_{\rho_j}>0 \] In fact, the eigenvectors corresponding to the repeated eigenvalues are given by the formula \[ \xi=(\xi_1,\xi_2,\dots,\xi_n,\xi_{n+1},\xi_{n+2},-\frac{N_1}{N_3}\xi_{n+1}-\frac{N_2}{N_3}\xi_{n+2},-\frac{1}{p_\theta}\sum_{i=1}^n\xi_ip_{\rho_i}),~\textnormal{for}~(\xi_1,\dots,\xi_{n+2})\in\mathbb{R}^{n+2} \] and thus the dimension of the eigenspace is $n+2$.

The reader should notice that hyperbolicity is not valid at $\rho =0$ (as $\det A$ vanishes) but strict hyperbolicity still holds when some of the $\rho_i$'s vanish provided $\rho \ne 0$.


\section{The zero-diffusion limit to multicomponent non-isothermal flows}\label{sec5}

As in section \ref{sec3}, we consider the hyperbolic-parabolic system
\begin{equation} \label{conv-eq1}
    \del_t\rho_i+\textnormal{div}(\rho_iv)=\textnormal{div}(-\rho_iu_i)
\end{equation} \begin{equation} \label{conv-eq2}
    \del_t(\rho v)+\textnormal{div}(\rho v\otimes v+p\mathbb{I})=0
\end{equation} \begin{equation} \label{conv-eq3}
    \del_t\left(\rho e+\frac{1}{2}\rho v^2\right)+\textnormal{div}\left((\rho e+\frac{1}{2}\rho v^2)v+pv\right)=\textnormal{div}\left(\kappa\nabla\theta-\sum_i(\rho_ie_i+p_i)u_i\right) \, ,
\end{equation} 
where $u_i$ are determined by solving
\begin{align} \label{conv-eq5}
    -\sum_{j\not=i}b_{ij}\theta\rho_i\rho_j(u_i-u_j) &=\epsilon\left(-\frac{\rho_i}{\rho}\nabla p+\rho_i\theta\nabla\frac{\mu_i}{\theta}-\theta(\rho_ie_i+p_i)\nabla\frac{1}{\theta}\right)
\\
\label{conv-eq6}
    \sum_i\rho_iu_i&=0 \, ,
\end{align} 
and which is endowed with the dissipation structure
\begin{align} \label{conv-eq4}
    \del_t(-\rho\eta)+\textnormal{div}(-\rho\eta v)  &+\frac{1}{\theta^2}\kappa|\nabla\theta|^2 - \frac{1}{\theta}\sum_iu_i\cdot d_i
\\
\nonumber
     &=\textnormal{div}\left(\frac{1}{\theta}\sum_i(\rho_ie_i+p_i-\rho_i\mu_i)u_i-\frac{1}{\theta}\kappa\nabla\theta\right) \, .
\end{align}

Moreover, consider the system obtained when neglecting the mass diffusive effects ($\epsilon = 0$) but including heat conduction:
 \begin{equation} \label{conv-eq7}
    \del_t\rho_i+\textnormal{div}(\rho_iv)=0
\end{equation} \begin{equation} \label{conv-eq8}
    \del_t(\rho v)+\textnormal{div}(\rho v\otimes v+p\mathbb{I})=0
\end{equation} \begin{equation} \label{conv-eq9}
    \del_t\left(\rho e+\frac{1}{2}\rho v^2\right)+\textnormal{div}\left((\rho e+\frac{1}{2}\rho v^2)v+pv\right)=\textnormal{div}\left(\kappa\nabla\theta\right)
\end{equation} 
endowed with the limiting dissipation structure
\begin{equation} \label{conv-eq10}
    \del_t(-\rho\eta)+\textnormal{div}(-\rho\eta v)=\textnormal{div}\left(-\frac{1}{\theta}\kappa\nabla\theta\right)-\frac{1}{\theta^2}\kappa|\nabla\theta|^2 \, .
\end{equation} 

Local existence results for smooth solutions are available for multicomponent systems  (see \cite{GM} and \cite[Ch. 8]{GioMulti}). For hyperbolic-parabolic systems like the ones above with general initial data one can expect existence of a unique smooth solution, which however can break down at finite time and the time of existence in general depends on the diffusion constants $\eps$ and $\kappa$. More precisely, under sufficient conditions on the initial data, the diffusion coefficients and the free energy function, if $Q_T=\Omega\times(0,T)$, where $\Omega$ is bounded and $T>0$, there exists a $\tau\in(0,T]$ such that the above problem possesses a unique solution \[ (\rho_1,\dots,\rho_n)\in W^1_p(Q_\tau;\mathbb{R}_+^n),~v\in W^{2,1}_p(Q_\tau;\mathbb{R}^3),~ \theta\in W^{2,1}(Q_\tau;\mathbb{R}_+) \] where the spaces above are defined as follows: \[ W^{2,1}_p(Q_T)=\{u\in L^p(Q_T):\del_t^\beta\del_x^\alpha u\in L^p(Q_T) ~ \textnormal{for all} ~ 0<2\beta+|\alpha|\le2\} \] \[ W^{1,0}_p(Q_T)=\{u\in L^p(Q_T):\del_x^\alpha u\in L^p(Q_T) ~ \textnormal{for all} ~ |\alpha|=1\} \] For more details we refer to \cite{D}.

In this section, we show that smooth solutions of \eqref{conv-eq1}-\eqref{conv-eq6} converge to solutions of \eqref{conv-eq7}-\eqref{conv-eq10} as $\epsilon$ tends to zero, so long as the solutions of the latter remain in the smooth regime and the theory developed should be understood as indicating conditions for convergence of thermomechanical theories in the smooth regime.

In order to show the convergence from \eqref{dis-eq7} as $\varepsilon \to 0$ we use the following ingredients: 

(a) The method of relative entropy introduced in \cite{Daf}, here employed in the form proposed in \cite{CT} :
\begin{equation} \label{conv-eq16}
        H(U|\bar{U}) = H(U)  - H (\bar{U})-G(\bar{U})\cdot(A(U)-A(\bar{U}))  \, .
 \end{equation} 
By \cite[Appendix A]{CT}, whenever conditions (i)-(iv) for \eqref{dis-entreq} hold, the relative entropy can be written as $H(U|\bar{U}) = \hat{H}(A(U)|A(\bar{U}))$, where 
$\hat{H} (V) $ is a strictly  convex function.
Hence,  $H(U|\bar{U})$ vanishes if and only if $A(U)=A(\bar{U})$ (and by (i) if $U=\bar{U}$) and it can 
serve to measure the distance between two solutions. 

(b) A second ingredient is the control of diffusion by dissipation (see \cite{CT},\cite[Sec 4.6]{Dafbook}), that is  a hypothesis 
 that there exist constants $\nu_1>0$ and $\nu_2>0$ such that 
 \begin{equation} \label{conv-eq11}
    \sum_{\alpha,\beta}\nabla G(U)\partial_\alpha UB_{\alpha\beta}^\kappa(U)\partial_\beta U\ge\nu_1\sum_\alpha\left|\sum_\beta B_{\alpha\beta}^\kappa(U)\partial_\beta U\right|^2
\end{equation} \begin{equation} \label{conv-eq12}
    \sum_{\alpha,\beta}\nabla G(U)\partial_\alpha UB_{\alpha\beta}^\epsilon(U)\partial_\beta U\ge\nu_2\sum_\alpha\left|\sum_\beta B_{\alpha\beta}^\epsilon(U)\partial_\beta U\right|^2
\end{equation} 

Next, we list hypotheses used on the thermodynamic functions. For the internal energy, when expressed in the form $\rho e = \rho \tilde e (\rho_1, \dots, \rho_n, \rho \eta)$, we require
\begin{align}\label{hypo1}
\frac{\del (\rho \tilde e)}{\del (\rho \eta)} > 0 \, , \quad \frac{\del^2 (\rho \tilde e)}{\del (\rho \eta)^2} > 0
\tag {H$_1$}
\\
\label{hypo2}
\nabla^2_{(\rho_1, \dots, \rho_n, \rho \eta)} ( \rho \tilde e ) > 0
\tag{H$_2$}
\end{align}
Hypothesis \eqref{hypo1} is natural in thermodynamics stating that the temperature $\theta > 0$ and ensuring convexity of the energy as a function of entropy. Hypothesis \eqref{hypo2} implies that the system \eqref{dis-eq27}-\eqref{dis-eq29} is hyperbolic and excludes various interesting models related to pressure laws of Van-der-Waals type.

Using \eqref{hypo1} one may invert the equation $\theta = \frac{\del (\rho \tilde e)}{\del (\rho \eta)}$ and define the inverse function
$\rho \eta =  (\rho \eta)^* (\rho_1, ... , \rho_n, \theta)$. This yields the Legendre transform \cite[Sec 5]{Cal}
\begin{equation}
\label{eqlegendre}
\rho \psi = \rho \tilde e (\rho_1, \dots , \rho_n, (\rho \eta)^* )  - \theta (\rho \eta)^* \quad \mbox{where} \quad \frac{\del (\rho \tilde e)}{\del (\rho \eta)} (\rho_1, \dots , \rho_n, (\rho \eta)^*) = \theta
\end{equation}
A computation shows that \eqref{hypo1},  \eqref{hypo2} imply
\begin{equation}
\label{convexitypsi}
 \nabla^2_{(\rho_1, \dots , \rho_n)} (\rho \psi) > 0 \, , \quad (\rho \psi)_{\theta \theta}  < 0  \, .
\end{equation}
The latter should be compared to \eqref{dis-eq13} and property (iii).

An alternative is to define the Legendre transform through the direct formula
\begin{equation}\label{legendre}
(\rho \psi )(\rho_1, \dots , \rho_n, \theta) = \inf_{0 < \rho \eta < \infty} \big  \{ \rho \tilde e(\rho_1, \dots , \rho_n, \rho \eta) - \theta \rho \eta  \big \}
\end{equation}
Under \eqref{hypo1} and 
\begin{equation}
\label{hypo3}
\lim_{\rho \eta \to 0} \frac{\del (\rho \tilde e)}{\del (\rho \eta)} = 0 \, , \quad \lim_{\rho \eta \to \infty}  \frac{\rho \tilde e }{\rho \eta}  = \infty \, 
\end{equation}
this problem has at most one solution computed via \eqref{legendre}. An advantage of this approach is that
the convexity conditions  \eqref{convexitypsi} follow directly from the minimization formula \eqref{legendre}. On the other hand to solve \eqref{legendre}
requires the assumption \eqref{hypo3}. This formulation and \eqref{hypo3} is consistent with the third law of thermodynamics (that the entropy vanishes at 
the state of zero-temperature, see \cite[Sec 1.10]{Cal}) but the  popular model of the ideal gas (with constant heat capacity) violates the third law and presents negative entropies. Nevertheless, the relation \eqref{eqlegendre} between internal energy $\rho e$ and Helmholtz free energy $\rho \psi$ is still valid.

The adaptation of the general framework to the system \eqref{conv-eq1}-\eqref{conv-eq6} requires some computations, and it is remarkable
that such a complicated system hides a simple structure. Let
$U=(\rho_1,\dots,\rho_n,v,\theta)^\top$ be a solution of \eqref{conv-eq1}-\eqref{conv-eq4} and  $\bar{U}=(\bar{\rho}_1,\dots,\bar{\rho}_n,\bar{v},\bar{\theta})^\top$ of \eqref{conv-eq7}-\eqref{conv-eq10}. Using  \eqref{dis-entflux}, \eqref{dis-eq11} and \eqref{intro-const1}-\eqref{intro-const4} we arrive at
\begin{equation}\label{relen}
\begin{aligned}
    H(U|\bar{U}) & =  - \rho \eta + \bar \rho \bar \eta 
                                - \frac{1}{\bar \theta}  \sum_{j=1}^n   (  \bar \mu_j -\tfrac{1}{2} \bar v^2  ) ( \rho_j - \bar \rho_j) - \frac{ \bar v}{\bar \theta} \cdot (\rho v - \bar \rho \bar v ) 
\\
                                &\qquad + \frac{1}{\bar \theta}\Big ( \rho e + \tfrac{1}{2} \rho v^2 - \bar \rho \bar e -  \tfrac{1}{2} \bar \rho \bar v^2  \Big )       
\\
    &= \frac{1}{2\bar \theta} \rho | v - \bar v|^2  + \frac{1}{\bar \theta}  \Big [ \Big (   \rho \psi - \bar \rho \bar \psi - \sum_{j=1}^n \bar \mu_j (\rho_j - \bar \rho_j) + \bar \rho \bar \eta (\theta - \bar \theta) \Big ) 
     + (\theta - \bar \theta) ( \rho \eta - \bar \rho \bar \eta) \Big ]
\\
    & = \frac{1}{\bar{\theta}}\left( \frac{1}{2}\rho|v-\bar{v}|^2  +  J( \rho_1, \dots, \rho_n, \theta | \bar \rho_1, \dots, \bar \rho_n, \bar \theta) \right)
\end{aligned}
\end{equation} where we set $\omega = (\rho_1, \dots, \rho_n ,\theta)$, $\bar \omega = (\bar \rho_1 , \dots, \bar \rho_n, \bar \theta)$ and use \eqref{intro-const2} to write
\begin{align}
 \label{lemma-eq1}
   J(\omega | \bar \omega) &:= \rho\psi-\bar{\rho}\bar{\psi}-\sum_i\bar{\mu}_i(\rho_i-\bar{\rho}_i)+\bar{\rho}\bar{\eta} (\theta-\bar{\theta}) 
    + (\theta - \bar \theta) ( \rho \eta - \bar \rho \bar \eta)
    \\
  \nonumber
    &= \rho e-\bar{\rho}\bar{e}-\sum_i\bar{\mu}_i(\rho_i-\bar{\rho}_i)-\bar{\theta}(\rho\eta-\bar{\rho}\bar{\eta})
\\
\label{relativequant}
   I(U|\bar{U}) &:= \bar \theta H(U|\bar{U}) = \frac{1}{2} \rho | v - \bar v|^2 + J(\omega | \bar \omega) \, .
\end{align} 
Due to hypothesis \eqref{hypo2} and \eqref{energy-const}, the quantity $J(\omega | \bar \omega)$ will serve as a measure of the distance between the states
$\omega$ and $\bar\omega$, in analogy to the situation in single component fluids  \cite{Daf}. 
This suggests to calculate  the evolution of the quantity \eqref{relativequant}.

Subtracting the entropy identities \eqref{conv-eq10} from \eqref{conv-eq4} and multiplying by $\bar{\theta}$, we obtain \begin{equation} \label{conv-eq13}
    \begin{split}
        &\del_t(-\bar{\theta}\rho\eta+\bar{\theta}\bar{\rho}\bar{\eta})+\textnormal{div}(-\rho\eta v\bar{\theta}+\bar{\rho}\bar{\eta}\bar{v}\bar{\theta})= \textnormal{div}\left(\frac{\bar{\theta}}{\theta}\sum_j(h_j-\rho_j\mu_j)u_j-\frac{\bar{\theta}}{\theta}\kappa\nabla\theta\right) \\
        & - \del_t\bar{\theta}(\rho\eta-\bar{\rho}\bar{\eta})+\nabla\bar{\theta}\cdot(-\rho\eta v+\bar{\rho}\bar{\eta}\bar{v})-\frac{1}{\theta}\sum_j(h_j-\rho_j\mu_j)u_j\cdot\nabla\bar{\theta}+\frac{1}{\theta}\kappa\nabla\theta\cdot\nabla\bar{\theta} \\
        & - \frac{\bar{\theta}}{\theta^2}\kappa|\nabla\theta|^2+\frac{\bar{\theta}}{\theta}\sum_ju_j\cdot d_j+\textnormal{div}\left(\frac{\bar{\theta}}{\bar{\theta}}\bar{\kappa}\nabla\bar{\theta}\right)-\frac{1}{\bar{\theta}}\bar{\kappa}\nabla\bar{\theta}\cdot\nabla\bar{\theta}+\frac{\bar{\theta}}{\bar{\theta}^2}\bar{\kappa}|\nabla\bar{\theta}|^2
    \end{split}
\end{equation} Likewise, subtracting system \eqref{conv-eq7}-\eqref{conv-eq9} from system \eqref{conv-eq1}-\eqref{conv-eq3} and multiplying the result by $-\bar{\theta}G(\bar{U})$, where $G$ is the multiplier from \eqref{dis-eq11}, we obtain: \begin{equation} \label{conv-eq14}
\begin{split}
    & \del_t\left(\sum_i\left(\frac{1}{2}\bar{v}^2-\bar{\mu}_i\right)(\rho_i-\bar{\rho}_i)-\bar{v}(\rho v-\bar{\rho}\bar{v})+\left(\rho e+\frac{1}{2}\rho v^2-\bar{\rho}\bar{e}-\frac{1}{2}\bar{\rho}\bar{v}^2\right)\right) \\
    & +\textnormal{div}\Big(\sum_i\left(\frac{1}{2}\bar{v}^2-\bar{\mu}_i\right)(\rho_iv-\bar{\rho}_i\bar{v})-\bar{v}(\rho v\otimes v-\bar{\rho}\bar{v}\otimes\bar{v}+(p-\bar{p})\mathbb{I}) \\
    & + \left(\rho e+\frac{1}{2}\rho v^2+p\right)v-\left(\bar{\rho}\bar{e}+\frac{1}{2}\bar{\rho}\bar{v}^2+\bar{p}\right)\bar{v}\Big) = \sum_i\del_t\left(\frac{1}{2}\bar{v}^2-\bar{\mu}_i\right)(\rho_i-\bar{\rho_i}) \\
    & + \sum_i\nabla\left(\frac{1}{2}\bar{v}^2-\bar{\mu}_i\right)\cdot(\rho_iv-\bar{\rho}_i\bar{v})-\sum_i\left(\frac{1}{2}\bar{v}^2-\bar{\mu}_i\right)\textnormal{div}(\rho_iu_i)-\del_t\bar{v}\cdot(\rho v-\bar{\rho}\bar{v}) \\
    & - \rho v\nabla\bar{v}\cdot v+\bar{\rho}\bar{v}\nabla\bar{v}\cdot\bar{v}-(p-\bar{p})\textnormal{div}\bar{v}+\textnormal{div}\left(\kappa\nabla\theta-\sum_jh_ju_j\right)-\textnormal{div}(\bar{\kappa}\nabla\bar{\theta})
\end{split}
\end{equation}
where we used the abbreviation $h_j := \rho_j e_j + p_j$.

Next, we add equations \eqref{conv-eq13} and \eqref{conv-eq14} and perform a series of calculations detailed in Appendix \ref{appD} to re-organize the terms in the
right hand side. The resulting relative entropy identity reads:
\begin{equation} \label{conv-eq15}
\begin{split}
    & \del_tI(U|\bar{U})+\textnormal{div}\Big[vI(U|\bar{U})+(p-\bar{p})(v-\bar{v})+\sum_j\rho_ju_j(\mu_j-\bar{\mu}_j) \\
    & - (\theta-\bar{\theta})\left(\frac{1}{\theta}\kappa\nabla\theta-\frac{1}{\bar{\theta}}\bar{\kappa}\nabla\bar{\theta}\right)+\frac{1}{\theta}(\theta-\bar{\theta})\sum_j(h_j-\rho_j\mu_j)u_j\Big]+\bar{\theta}\kappa\left(\frac{\nabla\theta}{\theta}-\frac{\nabla\bar{\theta}}{\bar{\theta}}\right)^2 \\
    & - \frac{\bar{\theta}}{\theta}\sum_ju_j\cdot d_j=(\del_t\bar{\theta}+\bar{v}\cdot\nabla\bar{\theta})(-\rho\eta)(\omega|\bar{\omega})-p(\omega|\bar{\omega})\textnormal{div}\bar{v}-(\eta-\bar{\eta})\rho(v-\bar{v})\cdot\nabla\bar{\theta} \\
    & - \sum_j\nabla\bar{\mu}_j\left(\frac{\rho_j}{\rho}-\frac{\bar{\rho}_j}{\bar{\rho}}\right)\rho(v-\bar{v})-\rho(v-\bar{v})\nabla\bar{v}\cdot(v-\bar{v})-\sum_j\nabla\bar{\mu}_j\cdot\rho_ju_j \\
    & - \left(\frac{\nabla\theta}{\theta}-\frac{\nabla\bar{\theta}}{\bar{\theta}}\right)\frac{\nabla\bar{\theta}}{\bar{\theta}}(\bar{\theta}\kappa-\theta\bar{\kappa})-\frac{1}{\bar{\theta}}\nabla\bar{\theta}\sum_j(h_j-\rho_j\mu_j)u_j
\end{split} 
\end{equation} where
 \begin{align*}
  p(\omega|\bar{\omega}) &=p-\bar{p}-\sum_j\bar{p}_{\rho_j}(\rho_j-\bar{\rho}_j)-\bar{p}_\theta(\theta-\bar{\theta})
  \\
  (-\rho\eta)(\omega|\bar{\omega}) &=-\rho\eta+\bar{\rho}\bar{\eta}+\sum_j(\bar{\rho}\bar{\eta})_{\rho_j}(\rho_j-\bar{\rho}_j)+(\bar{\rho}\bar{\eta})_\theta(\theta-\bar{\theta}) 
  \end{align*}
 
In the sequel, let $\mathcal{U} \subset  (\mathbb{R}^+)^{n+1}$ be a set in the positive cone $(\mathbb{R}^+)^{n+1}$ with $\overline{\mathcal{U}}$ compact,
and suppose that states $\omega,\bar{\omega}\in \mathcal{U}$  satisfy
   \begin{equation} \label{ass1}
        0<\rho_j,\bar{\rho}_j\le M
    \end{equation} \begin{equation} \label{ass2}
        0<\delta\le\rho,\bar{\rho}\le M
    \end{equation}\begin{equation} \label{ass3}
        0<\delta\le\theta,\bar{\theta}\le M
    \end{equation}
for some $\delta , M > 0$.

\begin{lemma}\label{lemma1}
Let $\omega,\bar{\omega}\in \mathcal{U}$ satisfy \eqref{ass1}-\eqref{ass3} and suppose that $\psi(\rho_1,\dots,\rho_n,\theta)\in C^3(\overline{\mathcal{U}})$ 
satisfies the thermodynamic relations \eqref{intro-const1}-\eqref{intro-gd} and the hypotheses \eqref{hypo1},\eqref{hypo2} (and thus \eqref{convexitypsi}).
There exist constants $c_1,c_2,c_3>0$ depending on $\delta, M$, such that for $\omega,\bar{\omega}\in\mathcal{U}$ we have
 \begin{equation} \label{conv-eq17}
    c_1|\omega-\bar{\omega}|^2\le J (\omega | \bar \omega)
\end{equation}
\begin{equation} \label{conv-eq19}
    |p(\omega |\bar{\omega})|\le c_2 J (\omega | \bar \omega)
\end{equation}
\begin{equation} \label{conv-eq20}
    |(-\rho\eta)(\omega |\bar{\omega})|\le c_3 J (\omega | \bar \omega) \, .
\end{equation}

\end{lemma} 
\begin{proof}
Consider the form \eqref{lemma-eq1} and use \eqref{energy-const} and the
convexity of  $\rho e=\rho\tilde{e}(\rho_1,\dots,\rho_n,\rho\eta)$ in the variables $(\rho_1,\dots,\rho_n,\rho\eta)$ to obtain
 \begin{equation} \label{lemma-eq2} \begin{split}
   J(\omega | \bar \omega)
        & = \rho\tilde{e}-\overline{\rho\tilde{e}}-\sum_i\overline{\left(\frac{\del\rho\tilde{e}}{\del\rho_i}\right)}(\rho_i-\bar{\rho}_i)-\overline{\left(\frac{\del\rho\tilde{e}}{\del\rho\eta}\right)}(\rho\eta-\bar{\rho}\bar{\eta}) \\
        & \ge c\left(\sum_i|\rho_i-\bar{\rho}_i|^2+|\rho \eta - \bar \rho \bar \eta|^2\right)
    \end{split}
    \end{equation} where $c=\inf_{(\rho_1,\dots,\rho_n,\rho\eta)\in \overline{\mathcal{U}}}\nabla^2_{(\rho_1,\dots,\rho_n,\rho\eta)}\rho\tilde{e}>0$, for $0<\delta\le\rho<M$. 
    Next, the map
    $(\rho_1, \dots , \rho_n, \theta) \mapsto (\rho_1, \dots , \rho_n, \rho \eta)$ defined by $\eta (\rho_1, \dots , \rho_n, \theta) = - \psi_\theta $ can be inverted on the set $\overline{\mathcal{U}}$
    and since $\frac{\del (\rho \eta)}{\del \theta} =\frac{1}{\theta}\rho c_v > 0$, for $\rho,\theta>0$, the inverse map is Lipschitz and 
\begin{equation} \label{lemma-eq3}
        |\theta-\bar{\theta}|^2 + \sum_i|\rho_i-\bar{\rho}_i|^2\le C \left ( |\rho\eta-\bar{\rho}\bar{\eta}|^2+\sum_i|\rho_i-\bar{\rho}_i|^2 \right )
    \end{equation} 
where $C$ depends on $\min_{\overline{\mathcal{U}}} \frac{\del (\rho \eta)}{\del \theta}$.  Combining \eqref{lemma-eq2} with \eqref{lemma-eq3} gives \eqref{conv-eq17}.

The bounds \eqref{conv-eq19}-\eqref{conv-eq20} follow from the Taylor theorem, which provides 
\begin{equation}\label{lemma-eq4}
        p(\omega|\bar{\omega})=(\omega-\bar{\omega})^\top\left(\int_0^1\int_0^s\nabla^2p(\tau \omega+(1-\tau)\bar{\omega})d\tau ds\right)(\omega-\bar{\omega})
    \end{equation} 
    \begin{equation}\label{lemma-eq5}
        (-\rho\eta)(\omega|\bar{\omega})=(\omega-\bar{\omega})^\top\left(\int_0^1\int_0^s\nabla^2(-\rho\eta)(\tau \omega+(1-\tau)\bar{\omega})d\tau ds\right)(\omega-\bar{\omega})
    \end{equation}
and the regularity of $\psi$ implies that $p,\rho\eta \in C^2 (\overline{\mathcal{U}})$.
\end{proof}

Let $\mathbb{T}^3=(\mathbb{R}/2\pi\mathbb{Z})^3$ be the three-dimensional torus. We have the following convergence result:

\begin{theorem}\label{thm1}
    Let $\bar{U}^\kappa$ be a classical solution of \eqref{conv-eq7}-\eqref{conv-eq9} and 
    let $U^{\epsilon,\kappa}$ be a family of classical solutions of \eqref{conv-eq1}-\eqref{conv-eq3} defined on $\mathbb{T}^3\times[0,T]$ for
    some $T < T^*$, which emanate from smooth data $\bar{U}_0^\kappa$, $U_0^{\epsilon,\kappa}$, respectively, and satisfy
   the uniform bounds \eqref{ass1}-\eqref{ass3} for $\delta,M>0$. 
Moreover, assume that $\psi\in C^3(\bar{\mathcal{U}})$ satisfies \eqref{hypo1},  \eqref{hypo2}, 
that
\begin{equation} \label{ass4}
        \sum_j\left|\frac{h_j-\rho_j\mu_j}{\sqrt{\rho_j}}\right|^2\le \alpha(\rho e+1)
    \end{equation} 
for some $\alpha>0$ and that  $0\le\kappa(\rho_1,\dots,\rho_n,\theta)\le M$.
Then, there exist constants $c,C>0$ depending on $\delta, M, \alpha$ but otherwise independent of $\epsilon$ such that 
\begin{equation} \label{conv-eq22}
    \begin{split}
        \int_{\mathbb{T}^3}I(U^{\epsilon,\kappa}|\bar{U}^\kappa)dx & \le c\int_{\mathbb{T}^3}I(U^{\epsilon,\kappa}_0|\bar{U}^\kappa_0)dx+\epsilon C
    \end{split}
\end{equation} In particular, if $\int_{\mathbb{T}^3}I(U^{\epsilon,\kappa}_0|\bar{U}^\kappa_0)dx\to0$ as $\epsilon\to0$, then \begin{equation} \label{conv-eq23}
    \sup_{t\in(0,T)}\int_{\mathbb{T}^3}I(U^{\epsilon,\kappa}|\bar{U}^\kappa)dx\to0 ~~ \textnormal{as $\epsilon\to0$}.
\end{equation} 
\end{theorem} 

\begin{proof} We first integrate the relative entropy identity \eqref{conv-eq15} to obtain \begin{equation} \label{conv-eq24}
    \begin{split}
        & \frac{d}{dt} \int_{\mathbb{T}^3} I(U^\varepsilon|\bar{U})dx+\int_{\mathbb{T}^3}\bar{\theta}\kappa\left(\frac{\nabla\theta}{\theta}-\frac{\nabla\bar{\theta}}{\bar{\theta}}\right)^2dx-\int_{\mathbb{T}^3}\frac{\bar{\theta}}{\theta}\sum_ju_j\cdot d_jdx \\
        & = \int_{\mathbb{T}^3}(\del_t\bar{\theta}+\bar{v}\cdot\nabla\bar{\theta})(-\rho\eta)(U^\varepsilon|\bar{U})dx-\int_{\mathbb{T}^3}p(U^\varepsilon|\bar{U})\textnormal{div}\bar{v}dx \\
        & + \int_{\mathbb{T}^3}(\eta-\bar{\eta})\rho(v-\bar{v})\cdot\nabla\bar{\theta}dx-\int_{\mathbb{T}^3}\sum_j\nabla\bar{\mu}_j\left(\frac{\rho_j}{\rho}-\frac{\bar{\rho}_j}{\bar{\rho}}\right)\rho(v-\bar{v})dx \\
        & - \int_{\mathbb{T}^3}\rho(v-\bar{v})\nabla\bar{v}\cdot(v-\bar{v})dx-\int_{\mathbb{T}^3}\sum_j\nabla\bar{\mu}_j\cdot\rho_ju_jdx \\
        & - \int_{\mathbb{T}^3}\left(\frac{\nabla\theta}{\theta}-\frac{\nabla\bar{\theta}}{\bar{\theta}}\right)\frac{\nabla\bar{\theta}}{\bar{\theta}}\kappa(\bar{\theta}-\theta)dx-\int_{\mathbb{T}^3}\frac{1}{\bar{\theta}}\nabla\bar{\theta}\sum_j(h_j-\rho_j\mu_j)u_jdx \\
        & =: I_1+\cdots+I_8 \, .
    \end{split}
\end{equation} 
Our strategy is to control terms $I_1$ to $I_5$ by the integral of $I(U| \bar U)$ and terms $I_6$ to $I_8$ by the dissipation on the left-hand side. In particular, to control $I_1$ and $I_2$, we use \eqref{conv-eq20} and \eqref{conv-eq19} respectively. For $I_3$, using the regularity of $\psi$, and thus $\eta$, we have \[ \begin{split}
    (\eta-\bar{\eta})\rho(v-\bar{v})\cdot\nabla\bar{\theta} & \le c\left(\rho|\eta-\bar{\eta}|^2+\rho|v-\bar{v}|^2\right) \\
    & \le c\left(\sum_j|\rho_j-\bar{\rho}_j|^2+|\theta-\bar{\theta}|^2+\rho|v-\bar{v}|^2\right)
\end{split} \] Regarding $I_4$, we have \[ \begin{split}
    \sum_j\nabla\bar{\mu}_j\left(\frac{\rho_j}{\rho}-\frac{\bar{\rho}_j}{\bar{\rho}}\right)\rho(v-\bar{v}) & \le c\left(\rho\sum_j\left(\frac{\rho_j}{\rho}-\frac{\bar{\rho}_j}{\bar{\rho}}\right)^2+\rho(v-\bar{v})^2\right) \\
    & \le c\left(\sum_j(\rho_j-\bar{\rho}_j)^2+\rho(v-\bar{v})^2\right)
\end{split} \] since the map $f_j:(\rho_1,\dots,\rho_n)\mapsto\frac{\rho_j}{\rho_1+\cdots+\rho_n}$, for $j\in\{1,\dots,n\}$ is Lipschitz under the assumption $0<\delta\le\rho$. Indeed \[ \frac{\del f_j}{\del\rho_i}=\begin{cases}
     -\frac{\rho_j}{\rho^2}, & i\not=j \\
    \frac{\rho-\rho_j}{\rho^2}, & i=j  
\end{cases} \] hence \[ \left|\frac{\del f_j}{\del\rho_i}\right|\le\frac{c}{\rho}\le \frac{c}{\delta} \, . \] 
As for $I_5$ \[ \rho(v-\bar{v})\nabla\bar{v}\cdot(v-\bar{v})\le c\rho|v-\bar{v}|^2 \]

Now, if we set $\rho_ju_j=\epsilon\rho_j\tilde{u}_j$ we have, by Young's inequality and for $\mu=\min_{i\not=j}b_{ij} > 0$:
 \[ \begin{split}
    I_6 & \le \frac{\epsilon}{\mu\delta}\int_{\mathbb{T}^3}\sum_j\frac{\rho_j}{\rho}|\nabla\bar{\mu}_j|^2dx+\frac{\epsilon\mu\delta}{4}\int_{\mathbb{T}^3}\rho\sum_j|\sqrt{\rho_j}\tilde{u}_j|^2dx \\
    & \le \epsilon C+\frac{\epsilon\mu\delta}{4}\int_{\mathbb{T}^3}\rho\sum_j|\sqrt{\rho_j}\tilde{u}_j|^2dx=:I_{61}+I_{62} \end{split} \] Likewise, using assumption \eqref{ass4}, we have that \[ \begin{split}
    I_8 & \le \frac{\epsilon}{\mu\delta}\int_{\mathbb{T}^3}\frac{1}{\rho}\sum_j\left|\frac{h_j-\rho_j\mu_j}{\sqrt{\rho_j}}\right|^2\left|\frac{\nabla\bar{\theta}}{\bar{\theta}}\right|^2dx+\frac{\epsilon\mu\delta}{4}\int_{\mathbb{T}^3}\rho\sum_j|\sqrt{\rho_j}\tilde{u}_j|^2dx \\
    & \le \epsilon C+\frac{\epsilon\mu\delta}{4}\int_{\mathbb{T}^3}\rho\sum_j|\sqrt{\rho_j}\tilde{u}_j|^2dx=:I_{81}+I_{82}
\end{split} \] Finally \[ \begin{split}
    I_7 & \le \frac{1}{2}\int_{\mathbb{T}^3}\bar{\theta}\kappa\left(\frac{\nabla\theta}{\theta}-\frac{\nabla\bar{\theta}}{\bar{\theta}}\right)^2dx+\frac{1}{2}\int_{\mathbb{T}^3}\kappa\frac{|\nabla\bar{\theta}|^2}{\bar{\theta}^3}(\theta-\bar{\theta})^2dx \\
    & \le \frac{1}{2}\int_{\mathbb{T}^3}\bar{\theta}\kappa\left(\frac{\nabla\theta}{\theta}-\frac{\nabla\bar{\theta}}{\bar{\theta}}\right)^2dx+c\int_{\mathbb{T}^3}(\theta-\bar{\theta})^2dx=: I_{71}+I_{72}
\end{split} \] 
Using \eqref{conv-eq17} we conclude that \[ I_1+\cdots+I_5+I_{72}\le c\int_{\mathbb{T}^3}I(U^\varepsilon|\bar{U})dx \, . \]

The error terms $I_{62}$, $I_{71}$ and $I_{82}$ are controlled by the dissipation on the left-hand side of \eqref{conv-eq24}. This is due to the assumption that $\delta\le\bar{\theta}$ and the following estimate, in which we use \eqref{dis-eq24} and \eqref{dis-eq22}: 
\[ \begin{split}
    -\frac{1}{\theta}\sum_ju_j\cdot d_j & = \epsilon\rho\sum_{i,j}M_{ij}\sqrt{\rho_i}\tilde{u}_i\cdot\sqrt{\rho_j}\tilde{u}_j \\
    & \ge \epsilon\rho\mu|\mathbb{P}_L\sqrt{\rho}\tilde{u}|^2
\end{split} \] where $\sqrt{\rho}\tilde{u}$ is the vector with components $\sqrt{\rho_i}\tilde{u}_i$, whose projection is computed as \[ \begin{split}
    \mathbb{P}_L\sqrt{\rho}\tilde{u} & = \sum_k(\mathbb{P}_L)_{ik}\sqrt{\rho_k}\tilde{u}_k \\
    & = \sqrt{\rho_i}\tilde{u}_i-\frac{\sqrt{\rho_i}}{\rho}\sum_k\rho_k\tilde{u}_k \\
    & = \sqrt{\rho_i}\tilde{u}_i
\end{split} \] Putting everything together we obtain: 
\begin{equation} \label{conv-eq25}
    \begin{split}
     \frac{d}{dt}    \int_{\mathbb{T}^3} I(U^\varepsilon|\bar{U})dx & + \frac{\delta}{2}\int_{\mathbb{T}^3}\kappa\left(\frac{\nabla\theta}{\theta}-\frac{\nabla\bar{\theta}}{\bar{\theta}}\right)^2dx+\frac{\delta\epsilon\mu}{2}\int_{\mathbb{T}^3}\rho\sum_j|\sqrt{\rho_j}\tilde{u}_j|^2dx \\
        & \le c\int_{\mathbb{T}^3}I(U^\varepsilon|\bar{U})dx+\epsilon C
    \end{split}
\end{equation} for appropriate constants $c,C>0$ independent of $\epsilon$. The dissipation terms are neglected and
we obtain the differential inequality
\begin{equation} \label{conv-eq27}
    \frac{d\varphi^\varepsilon(t)}{dt}\le c\varphi^\varepsilon(t)+\epsilon C
\end{equation}
for
 \[ \varphi^\varepsilon(t)=\int_{\mathbb{T}^3}I(U^\varepsilon(x,t)|\bar{U}(x,t))dx
  \] 
 Then, \eqref{conv-eq22} follows by Gr\"onwall's Lemma. \end{proof}

A particular case of the above analysis is the convergence to the adiabatic theory, i.e. when also $\kappa=0$. Consider the hyperbolic-parabolic system \eqref{conv-eq1}-\eqref{conv-eq4} with the linear system \eqref{conv-eq5} and the constraint \eqref{conv-eq6} and its hyperbolic counterpart \eqref{dis-eq27}-\eqref{dis-eq30}. Then, following the same process as before, we can obtain that the hyperbolic-parabolic system \eqref{conv-eq1}-\eqref{conv-eq4} converges as $\epsilon,\kappa\to0$ to the hyperbolic system without diffusion and heat conduction \eqref{dis-eq27}-\eqref{dis-eq30}: 

\begin{theorem}\label{thm2} 
Let $\bar{U}$ be a classical solution of \eqref{dis-eq27}-\eqref{dis-eq29} defined on a maximal interval of existence $\mathbb{T}^3\times[0,T^*)$ and let $U^{\epsilon,\kappa}$ be a family of classical solutions of \eqref{conv-eq1}-\eqref{conv-eq3} defined on $\mathbb{T}^3\times[0,T]$, for $T<T^*$, emanating from smooth data $\bar{U}_0,U_0^{\epsilon,\kappa}$ respectively. Under the hypotheses of theorem \ref{thm1}, there exist constants $c,C_1,C_2>0$ independent of $\epsilon$ and $\kappa$ such that \begin{equation} \label{conv-eq32}
\begin{split}
    \int_{\mathbb{T}^3}I(U^{\epsilon,\kappa}|\bar{U})dx & \le c\int_{\mathbb{T}^3}I(U^{\epsilon,\kappa}_0|\bar{U}_0)dx+\|\kappa\|_\infty C_1+\epsilon C_2
\end{split}
\end{equation} In particular, if $\int_{\mathbb{T}^3}I(U^{\epsilon,\kappa}_0|\bar{U}_0)dx\to0$ as $\epsilon,\kappa\to0$, then \begin{equation} \label{conv-eq33}
    \sup_{t\in(0,T)}\int_{\mathbb{T}^3}I(U^{\epsilon,\kappa}|\bar{U})dx\to0 ~~ \textnormal{as $\epsilon,\kappa\to0$}.
\end{equation} 
\end{theorem} 
\begin{proof}
    The proof is identical to the proof of Theorem \ref{thm1}, with the only difference that $\bar{\kappa}=0$ (whereas in the previous case $\kappa=\bar{\kappa}$) and thus $I_7$ is controlled by: \[ \begin{split}
    I_7 & = -\int_{\mathbb{T}^3}\kappa\left(\frac{\nabla\theta}{\theta}-\frac{\nabla\bar{\theta}}{\bar{\theta}}\right)\cdot\nabla\bar{\theta}  \\
    & \le \frac{1}{2}\int_{\mathbb{T}^3}\bar{\theta}\kappa\left(\frac{\nabla\theta}{\theta}-\frac{\nabla\bar{\theta}}{\bar{\theta}}\right)^2dx+\frac{1}{2}\int_{\mathbb{T}^3}\kappa\frac{|\nabla\bar{\theta}|^2}{\bar{\theta}}dx \\
    & \le \frac{1}{2}\int_{\mathbb{T}^3}\bar{\theta}\kappa\left(\frac{\nabla\theta}{\theta}-\frac{\nabla\bar{\theta}}{\bar{\theta}}\right)^2dx+\|\kappa\|_\infty C_1
\end{split} \]
\end{proof}

Next, we present a commentary on the hypotheses for Theorems \ref{thm1} and \ref{thm2}. The goal here is to discuss
aspects of the theory of smooth solutions; the situation for weak solutions presents serious challenges. The hypotheses on the bounds
\eqref{ass2}, \eqref{ass3} reflect the loss of strict hyperbolicity (and even hyperbolicity) of the
model at $\rho =0$. Also, we would expect that such continuum models are not valid for
very large temperatures or temperature near zero.

The hypothesis $\psi \in C^3 (\bar U)$ is a drawback as it does not hold at $\rho_i = 0$ for realistic models and at the same time one would 
insist on the range \eqref{ass1} that guarantees some of the components may disappear as an outcome of interactions.
The main problematic term is $I_3$.
Realistic models are discussed in  section \ref{sec:multigas} dealing with the multicomponent ideal gas. 
The reader will notice that $\psi \in C^3 (\bar U)$ holds on the restricted range
 \begin{equation} \label{ass1pr}
        0< \delta \le \rho_j,\bar{\rho}_j\le M
    \end{equation}
and gives convergence for solutions taking values in that range \eqref{ass1pr}, \eqref{ass2}, \eqref{ass3}.

The need concerning \eqref{ass4} originates from a deficiency of the models discussed here.
For the general model,  the partial energies $e_i$, pressures $p_i$ and enthalpies $h_i$ are not determined from the
 total free energy $\rho \psi$ via \eqref{const-final}, yet they enter the balance equations \eqref{intro-eq4}-\eqref{intro-eq7}
and thus extra constitutive relations have to be supplied. Hypothesis \eqref{ass4} concerns these extra constitutive relations. 
For the  case of simple mixtures, $\rho \psi = \sum_i \rho_i \psi_i(\rho_i, \theta)$, the need for extra constitutive relations does not arise;
one computes using \eqref{appB-eq12}-\eqref{appB-eq15} that 
$\frac{1}{\sqrt{\rho_j}} (h_j - \rho_j \mu_j) = - \theta \sqrt{\rho_j} \frac{ \del{\psi_j} }{\del \theta}$
and the terms are bounded  from \eqref{ass1}-\eqref{ass3}.

\section{Multicomponent ideal gases}\label{sec:multigas}
We present the constitutive model of an ideal multicomponent gas and compute the relative
constitutive functions that appear in the relative entropy formula \eqref{conv-eq15}. We refer to Callen \cite[Sec 13]{Cal} and 
Giovangigli \cite[Sec 6]{GioMulti} for details on the multicomponent ideal gas laws. Here, we outline a constitutive model defined
in terms of densities $\rho_i$ and temperature $\theta$. The model is a simple mixture of ideal gases
where the free energy of each component is given by $\rho_i \psi_i = R_i \theta \rho_i \log \rho_i  - c_i \rho_i \theta \log \theta$,
where $R_i > 0$ is the engineering gas constant and $c_i > 0$ the constant heat capacity of the $i$-th component, and the mixture free energy is
\begin{equation}
\label{feig}
\rho \psi = \sum_i \rho_i \psi_i = \sum_i R_i \theta \rho_i \log \rho_i  -  \sum_i c_i \rho_i \theta \log \theta \, .
\end{equation}
We then have
\begin{equation}
\begin{aligned}
\mu_j &= \frac{\del (\rho \psi)}{\del \rho_j} = R_j (1 + \log \rho_j) \theta - c_j \theta \log \theta
\\
\rho \eta &= - \frac{\del (\rho \psi)}{\del \theta} = - \sum_i R_i \rho_i \log \rho_i  + \sum_i c_i \rho_i (1 +  \log \theta)
\\
\rho e &= \sum_i c_i \rho_i \theta
\\
p &= - \rho \psi + \sum_j \rho_j \mu_j = \sum_j R_j \rho_j \theta
\end{aligned}
\end{equation}

The relative quantities are computed as follows: Using \eqref{lemma-eq1}, we have
\begin{equation}
\label{Jig}
\begin{aligned}
J(\omega | \bar \omega) 
&= \bar \theta \sum_i R_i \Big ( \rho_i \log \rho_i - \bar \rho_i \log \bar \rho_i - (1 + \log \bar \rho_i) (\rho_i - \bar \rho_i ) \Big ) 
\\
&\qquad + 
\bar \theta \sum_i c_i \rho_i  \Big ( - \log \theta + \log \bar \theta + \frac{1}{\bar \theta} (\theta - \bar \theta) \Big )
\\
&= \bar \theta \sum_i R_i (x \log x) \big (\rho_i | \bar \rho_i \big ) + \sum_i c_i \rho_i \bar \theta (- \log y)(\theta | \bar \theta)
\end{aligned}
\end{equation}
where $(x \log x) \big (\rho_i | \bar \rho_i \big ) $ is the quadratic part of the Taylor expansion of $x \log x$ and ditto for
$(- \log y)(\theta | \bar \theta)$. Due to the convexity of $(x \log x)$ and $(-\log y)$ both terms are positive.

Similarly, we compute
\begin{equation}
\label{pig}
p(\omega | \bar \omega) = \sum_j R_j  (\rho_j - \bar \rho_j ) (\theta - \bar \theta)
\end{equation}
and
\begin{equation}
\label{etaig}
\begin{aligned}
(-\rho \eta) (\omega | \bar \omega) &= -\rho \eta + \bar \rho \bar \eta + \sum_j  \overline{\frac{\del (\rho \eta)}{\del \rho_j }} (\rho_j - \bar \rho_j)
+ \overline{\frac{\del (\rho \eta)}{\del \theta}} (\theta - \bar \theta)
\\
&=  \sum_i R_i \big (x \log x \big ) \big (\rho_i | \bar \rho_i \big ) + \sum_i c_i \rho_i (- \log y)(\theta | \bar \theta)
 - \sum_i \frac{c_i}{\bar \theta} (\rho_i - \bar \rho_i) (\theta - \bar \theta)
\end{aligned}
\end{equation}
Finally observe that the relative entropy \eqref{relativequant} takes the form
\begin{equation}
\label{relenig}
I(U | \bar U) = \tfrac{1}{2} \rho | v - \bar v|^2 + \bar \theta \sum_i R_i \big (x \log x \big ) \big (\rho_i | \bar \rho_i \big ) 
+ \bar \theta   \sum_i c_i \rho_i \big(-\log y \big ) \big (\theta | \bar \theta \big ) 
\end{equation}

Using the explicit formulas \eqref{Jig}, \eqref{pig} and \eqref{etaig} the facts  $(x \log x)^{\prime \prime} = \frac{1}{x}$ and 
$(-\log x)^{\prime \prime} =\frac{1}{x^2}$ we obtain an analog of Lemma \ref{lemma1} for the
ideal gas:

\begin{lemma}\label{lemma2}
Let $\omega,\bar{\omega}\in \mathcal{U}$ satisfy \eqref{ass1}-\eqref{ass3} and let $\rho \psi$ 
be given in \eqref{feig}.
There exist constants $c_1,c_2>0$ depending on $\delta, M$, such that 
 \begin{equation} \label{multiig-eq17}
    c_1|\omega-\bar{\omega}|^2\le J (\omega | \bar \omega)
\end{equation}
\begin{equation} \label{multiig-eq19}
    |p(\omega |\bar{\omega})| +  |(-\rho\eta)(\omega |\bar{\omega})|  \le c_2J (\omega | \bar \omega)  \, .
\end{equation}
\end{lemma}

The missing element to conclude the proof of an analog of Theorem \ref{thm1} for the multicomponent ideal gas is an
estimate of the type
 \[ \begin{split}
  |  (\eta-\bar{\eta})\rho(v-\bar{v})\cdot\nabla\bar{\theta} |  \le C \Big ( \frac{1}{2} \rho |v - \bar v|^2 + J(\omega| \bar \omega) \Big )
\end{split} \] 
However, such an estimate is not valid for solutions that take value on the range \eqref{ass1}-\eqref{ass3}.


\begin{appendix}
\section{Equilibrium Thermodynamics}\label{appA}

Following \cite[sec 2,3,5]{Cal} we describe some elements of the thermodynamic theory in equilibrium, known as thermostatics, which seeks to describe the equilibrium states to which systems eventually evolve.  It is assumed that there exist equilibrium states of simple systems and that they are characterized by the internal energy of the system $E$, the volume $V$ and the mole number 
of the components of the system $M_1,\dots,M_n$. These quantities are called extensive parameters. It is postulated that there exists a function $H$ of the extensive parameters, called entropy, defined for all equilibrium states, so that the values assumed by the extensive parameters in the absence of internal constraints, are those that maximize the entropy over the manifold of constrained equilibria.  We refer to \cite[sec 2]{Cal}  for the precise statements and details. 
Note that the existence of the entropy is postulated only for equilibrium states, that the entropy of the system is postulated to be  provided by an extremum principle  and 
to be known as a function of the extensive parameters,
\begin{equation} \label{entrfund}
    H=H(E,V,M_1,\dots,M_n) \, .
\end{equation} 
This function, known as the  fundamental relation,  contains all thermodynamic information about the system. As the entropy of a composite system is additive over the constituent subsystem,
this function has to be homogeneous of first order,
 \begin{equation} \label{entrhom}
    H(\lambda E,\lambda V,\lambda M_1,\dots,\lambda M_n)=\lambda H(E,V,M_1,\dots,M_n),~\forall \lambda > 0, 
\end{equation} 
it is assumed to be differentiable, and  is postulated to be an increasing function of the internal energy.
As a result, the fundamental relation can be inverted with respect to the energy leading to a differentiable function of $H,V,M_1,\dots,M_n$ and an alternative form of the fundamental relation \begin{equation} \label{enerfund}
    E=E(H,V,M_1,\dots,M_n) \, ,
\end{equation} 
for which homogeneity of first order still holds.

With the aforementioned postulates we are ready to proceed to the thermodynamic analysis. Since the energy is a quantity that is easier to understand intuitively and also to measure in experiments, we prefer to use the energy representation of the fundamental relation, instead of the entropy one. In this case, the extremum principle is no longer the maximization of the entropy in equilibrium, but rather the minimization of the energy. The condition that guarantees the equivalence between the two principles is the fact that the entropy is an increasing function of the energy.

Now, the differential of the internal energy $E$ is:
\begin{equation} \label{appA-eq1}
    dE=\frac{\partial E}{\partial H}dH+\frac{\partial E}{\partial V}dV+\sum_{j=1}^n\frac{\partial E}{\partial M_j}dM_j
\end{equation} The partial derivatives in \eqref{appA-eq1} are called intensive parameters and occur so frequently that we introduce special symbols for them: \begin{equation} \label{const1}
    \frac{\partial E}{\partial H}=\theta, ~~ \textnormal{the temperature} \end{equation} \begin{equation} \label{const2}
    \frac{\partial E}{\partial V}=-p, ~~ \textnormal{the negative (total) pressure} \end{equation}
    \begin{equation} \label{const3}
    \frac{\partial E}{\partial M_i}=\mu_i, ~~ \textnormal{the chemical potential of the $i$-th component} \end{equation} With this notation, \eqref{appA-eq1} reads: \begin{equation} \label{appA-eq2}
    dE=\theta dH-pdV+\sum_j\mu_jdM_j
\end{equation} 
We briefly mention here that in the (equivalent) entropy representation the intensive parameters are \[ \frac{\partial H}{\partial E}=\frac{1}{\theta} ~,~ \frac{\partial H}{\partial V}=\frac{p}{\theta} ~,~ \frac{\partial H}{\partial M_i}=-\frac{\mu_i}{\theta} \] 

The first-order homogeneity of the fundamental equation \eqref{enerfund} provides an essential relation that will be used later in Appendix \ref{appB}. Differentiating \[ E(\lambda H,\lambda V,\lambda M_1,\dots,\lambda M_n)=\lambda E(H,V,M_1,\dots,M_n) \] with respect to $\lambda$ and then taking $\lambda=1$ gives \begin{equation} \label{appA-eq3}
    E=\theta H-pV+\sum_j\mu_jM_j
\end{equation} Equation \eqref{appA-eq3} is known as Euler equation, some authors refer to it as Gibbs-Duhem relation and we will retain this name here.

Back to the homogeneity of the fundamental relation, we can choose $\lambda=\frac{1}{V}$ and by introducing the partial mass densities $\rho_i=\frac{M_i}{V}$ and the total mass density 
$\rho=\frac{1}{V}$ (normalizing the total mass $\sum M_i=1$), the fundamental relation reads \begin{equation} \label{enerfund2}
    \rho e=\rho e(\rho\eta,\rho_1,\dots,\rho_n)
\end{equation} where $\rho e$ is the specific internal energy and $\rho\eta$ the specific entropy. Then \eqref{const1}-\eqref{const3} read 
\begin{equation} \label{energy-const}
    \frac{\partial(\rho e)}{\partial(\rho\eta)}=\theta ~,~ \frac{\partial(\rho e)}{\partial\rho_i}=\mu_i
\end{equation} and the Gibbs-Duhem relation takes the form
 \begin{equation} \label{gd}
    \rho e=\rho\eta\theta-p+\sum_j\rho_j\mu_j
\end{equation}

In thermodynamics, we often need to pass from a set of variables that contains only extensive parameters to another that contains some intensive parameters as well. For instance, throughout this paper, we select the temperature is one of the prime variables and we need to express the equations of state through this variable. To this end, we use partial Legendre transforms of the internal energy replacing some extensive variables by the corresponding intensive ones. These partial Legendre transforms are called thermodynamic potentials. A well-known example is the enthalpy: the enthalpy is the partial Legendre transform of the internal energy $E$ that replaces the volume $V$ by the pressure $p$ as an independent variable. Thus, the enthalpy $h$ is a function of $H,p,M_1,\dots,M_n$, defined by \begin{equation} \label{enthalpy}
    h=E+pV
\end{equation} where the replacement is achieved by first solving the relation \[ \frac{\partial E}{\partial V}=-p \] with respect to $V$ and then eliminating $V$ from \eqref{enthalpy}. In our case, we are interested in passing from the set $\{H,V,M_1,\dots,M_n\}$ to $\{\theta,V,M_1,\dots,M_n\}$, i.e. to replace the entropy by the temperature. The thermodynamic potential that allows us to do so is the Helmholtz free energy defined by \begin{equation} \label{free-en}
    \Psi=E-H\theta
\end{equation} The replacement is achieved by solving \[ \frac{\partial E}{\partial H}=\theta \] with respect to $H$ and eliminating $H$ from \eqref{free-en}. Using the notation of \eqref{enerfund2}, we have \begin{equation} \label{free-en2}
    \rho\psi=\rho\psi(\theta,\rho_1,\dots,\rho_n)
\end{equation}
 \begin{equation} \label{free-en-def}
    \rho\psi=\rho e-\rho\eta\theta
\end{equation} 
\begin{equation} \label{energy}
    e(\theta,\rho_1,\dots,\rho_n)=-\theta^2\left(\frac{\psi}{\theta}\right)_\theta
\end{equation} 
and the Gibbs-Duhem relation reads \begin{equation} \label{appA-eq4}
    \rho\psi+p=\sum_j\rho_j\mu_j \, .
\end{equation} 
Furthermore, one can use the relations \eqref{energy-const} to obtain 
\begin{equation} \label{const-final}
    \frac{\del(\rho\psi)}{\del\rho_i}=\mu_i ~,~ \frac{\del(\rho\psi)}{\del\theta}=-\rho\eta
\end{equation}

Finally, we note that the first postulate states that the entropy attains its maximum in equilibrium. This is translated as follows: the first variation of the entropy functional must vanish and the second one must be negative. The second condition in particular determines the stability of the predicted equilibrium states and suggests that the entropy is a concave function. A byproduct of this, is that the specific heat at constant volume $c_v:=e_\theta(\theta,\rho_1,\dots,\rho_2)$ is positive. Should we reformulate the stability criteria in the energy representation, we see that the energy attains its minimum in equilibrium, hence it is a convex function. These conditions can be extended to the thermodynamic potentials, as well. In particular, the Helmholtz free energy is a concave function of the temperature.


\section{Consistency with the Clausius-Duhem inequality}\label{appB}

In this appendix we outline the derivation of the Type-II model \eqref{intro-eq1}-\eqref{intro-eq3} for multicomponent flows, employed in this work, and derived in \cite{BDre}.
The framework is that of continuum mechanics, in which bodies are modeled as continuous media and their behavior is governed by (i) balance laws and (ii) constitutive relations 
that connect the various quantities and characterize the material response. For more details on modeling multicomponent fluids we refer to \cite{DP,GioMulti}.

The emphasis of this presentation and the development in \cite{BDre} is on the thermodynamic structure of the model so that consistency
is guaranteed with the Clausius-Duhem inequality. The model is using 
as primitive variables the mass densities, the velocities of the constituents and the temperature of the mixture. (The model in \cite{BDre} uses the specific entropy instead of
the temperature but this minor deviation does not cause major differences.)  A Type-II model is considered employing the following list of equations for the partial mass, partial momentum 
and (total) energy balance respectively: 
\begin{equation} \label{appB-eq1}
    \del_t\rho_i+\textnormal{div}(\rho_iv_i)=m_i
\end{equation} \begin{equation} \label{appB-eq2}
    \del_t(\rho_iv_i)+\textnormal{div}(\rho_iv_i\otimes v_i)=\textnormal{div}(S_i)+\rho_ib_i+f_i
\end{equation} \begin{equation} \label{appB-eq3}
    \begin{split}
        \del_t(\rho e+\sum_i\frac{1}{2}\rho_iv_i^2)+\textnormal{div}((\rho e & + \sum_i\frac{1}{2}\rho_iv_i^2)v)=\textnormal{div}(q-\sum_i(\rho_ie_i+p_i+\frac{1}{2}\rho_iv_i^2)u_i) \\
        & - \textnormal{div}(pv)+\textnormal{div}(v_i\cdot\sigma_i)+\rho b\cdot v+\rho r+\sum_i\rho_ib_i\cdot u_i
    \end{split}
\end{equation} where $\rho_i$ is the  mass density of the $i$-th component, $m_i$ the production of mass, $v_i$ the velocity, $S_i$ the partial stress tensor, $b_i$ the body force on the $i$-th component, $f_i$ the momentum production, $e_i$ the specific internal energy, $q_i$ the heat flux, $r_i$ the radiative heat supply and $\ell_i$ the internal energy production. 
Moreover, we define the total mass $\rho:=\sum_i\rho_i$, the barycentric velocity of the mixture $v$ such that $\rho v=\sum_i\rho_iv_i$. In \eqref{appB-eq3} we have also introduced the diffusional velocities $u_i:=v_i-v$, the thermal energy $\rho e:=\sum_i\rho_ie_i$, the total pressure $p:=\sum_ip_i$, total heat flux $q:=\sum_iq_i$, the net force applied to the mixture $\rho b:=\sum_i\rho_ib_i$ and the total radiative heat supply $\rho r:=\sum_i\rho_ir_i$.

Some remarks on the model: 

(a) There is a single temperature $\theta$ common to all components that is posiive.  For the thermodynamic quantities 
internal energy $e$, entropy $\eta$, and pressures $p_i$ we employ the relations \eqref{free-en2}, \eqref{free-en-def}, \eqref{appA-eq4} and \eqref{const-final}.

(b) The production of mass is often due to chemical reactions; here, we consider non-reactive fluids, i.e. we take $m_i=0$. For a treatment of the same model 
with chemical reactions we refer to \cite{BDre,GioMulti}. 

(c) The momentum productions $f_i$ model binary interactions between the species, e.g. due to friction. Likewise, the energy production $\ell_i$
is neglected due to (d), as we consider only the total energy equation \eqref{appB-eq3}, which is obtained by summing up all partial energy equations.

(d) Even though we allow for mass, momentum and energy production, we require that the total mass, total momentum and total energy be conserved, thus having the constraints \begin{equation} \label{appB-eq4}
    \sum_{i=1}^nm_i=\sum_{i=1}^nf_i=\sum_{i=1}^n\ell_i=0
\end{equation} where $n\in\mathbb{N}$ is the number of constituents.

(e) The consideration of the terms $b_i$ and $r_i$ is not essential, since both are external factors and can be modified accordingly (e.g. they can be taken equal to zero).

(f) The partial stresses are decomposed as $S_i=-p_i\mathbb{I}+\sigma_i$, where $-p_i\mathbb{I}$ is the elastic part of the stress tensor, with $p_i$ the partial pressures and $\mathbb{I}$ the identity tensor and $\sigma_i$ the viscous part. Here, for simplicity, we take $\sigma_i=0$ and thus $\textnormal{div}(S_i)$ reduces to $-\nabla p_i$. For the treatment of the theory with viscosity we refer to \cite{BDre}.

In order for our system to be thermodynamically complete, it needs to be consistent with the second law of thermodynamics expressed in the form of the Clausius-Duhem inequality: \begin{equation} \label{appB-eq5}
    \del_t(\rho\eta)+\textnormal{div}(\rho\eta v)\ge\textnormal{div}\Phi+\frac{\rho r}{\theta}
\end{equation} where $\rho\eta$ stands for the specific entropy, $\Phi$ for the entropy flux and $\theta$ for the temperature. The fact that we have an inequality signifies that there is an entropy production $\zeta$ that balances the two sides in \eqref{appB-eq5}. The entropy production is not a-priori specified. 
 Instead, the only information available is that according to the second law of thermodynamics, the entropy production must be non-negative. 
 As opposed to the balance laws of mass, momentum and energy, which determine the thermodynamic process from the assigned body force $b$, heat supply $r$, boundary and initial conditions,
  the Clausius-Duhem inequality plays the role of an admissibility criterion for thermodynamic processes that already comply with the balance laws. 
 In the case of one species the entropy flux is the heat flux divided by the temperature, however there is no available information for the multicomponent case, 
 and $\Phi$ is determined along with the entropy production, by the process of checking consistency with the Clausius-Duhem inequality.

For the thermodynamic reduction, one introduces the Helmholtz free energy $\rho\psi$, which inherits the relations \eqref{free-en2}-\eqref{const-final} as if we were in equilibrium and
using the entropy inequality \eqref{appB-eq5} and the balance laws \eqref{appB-eq1}-\eqref{appB-eq3}, one constructs an inequality known as dissipation inequality. Then, by controlling the body forces $b_i$ and heat supplies $r_i$ we can construct smooth processes that satisfy the mass, momentum and energy balances but attain at some point $(x,t)$ arbitrarily prescribed values for the primitive variables and their derivatives. Hence, the dissipation inequality is violated unless certain relations hold (see \cite{BDre}). From this process, one determines the entropy flux and 
entropy production 
\begin{equation} \label{appB-eq9}
    \Phi=\frac{q}{\theta}-\frac{1}{\theta}\sum_i(\rho_ie_i+p_i-\rho_i\mu_i)u_i
\end{equation} \begin{equation} \label{appB-eq10}
    \zeta=\frac{1}{\theta^2}q\cdot\nabla\theta-\frac{1}{\theta}\sum_iu_i\cdot\left(f_i-\nabla p_i+\rho_i\nabla\mu_i+\frac{1}{\theta}(\rho_ie_i+p_i-\rho_i\mu_i)\nabla\theta\right)
\end{equation} 

Equation \eqref{appB-eq10} is too complicated to allow for investigating necessary and sufficient conditions so that $ \zeta\ge0$ holds everywhere. 
However, following Onsager's reciprocal relations, one decouples the entropy production terms that correspond to different dissipative mechanisms and asks for linear dependence with a symmetric and positive semi-definite matrix of phenomenological coefficients,
According to this, a sufficient condition for the heat conduction term to be nonnegative is \begin{equation} \label{heat-cond} 
    q=\kappa\nabla\theta
\end{equation} with $\kappa=\kappa(\rho_1,\dots,\rho_n,\theta)\ge0$, which is the standard Fourier's law. Following a more sophisticated argument (originally due to Truesdell),
 one obtains the form for the momentum production \begin{equation} \label{mass-diff}
    f_i=\nabla p_i-\rho_i\nabla\mu_i-\frac{1}{\theta}(\rho_ie_i+p_i-\rho_i\mu_i)\nabla\theta-\theta\sum_{j\not=i}b_{ij}\rho_i\rho_j(v_i-v_j) \, , 
\end{equation} where $b_{ij}=b_{ij}(\rho_i,\rho_j,\theta)$ are symmetric, nonnegative binary-type interactions. 
For the full derivation we refer to \cite{BDre} section 7. 
After plugging the two closures in \eqref{appB-eq10} and using the symmetry of $b_{ij}$ the entropy production reads: \begin{equation} \label{appB-eq11}
    \zeta=\frac{1}{\theta^2}\kappa|\nabla\theta|^2+\frac{1}{2}\sum_{i=1}^n\sum_{j\not=i}b_{ij}\rho_i\rho_j|u_i-u_j|^2
\end{equation} 

An interesting comment is that even after determining the constitutive relations via the thermodynamic reduction, our model is not fully closed. Indeed, 
although we can compute the thermal energy from \eqref{energy} and the total pressure from the Gibbs-Duhem relation \eqref{appA-eq4}, 
it turns out that if we are just given the free energy density $\psi$ there is no way to determine the partial pressures $p_i$ and partial internal energies $e_i$. In other words, for a Type-II model it is not sufficient to be given solely the Helmholtz free energy, as it is for a Type-I model \cite{BDre}. Therefore, we require some more information: a first approach is to set up constitutive functions at the level of the components for $(e_i)_i$ and $(p_i)_i$. 
Since the total pressure and energy can be computed only $n-1$ constitutive functions need to be specified. 
Another approach that is often used is the so-called simple mixture theory.  In this simplification each component is assumed to have its own thermodynamical structure
(as if it were on its own), that is, it is described by its own free energy density $\psi_i=\psi_i(\rho_i,\theta)$ and the thermodynamic
functions of the component are determined via the relations: 
\begin{equation} \label{appB-eq12}
    e_i=e_i(\rho_i,\theta)=-\theta^2\left(\frac{\psi_i}{\theta}\right)_\theta
\end{equation} \begin{equation} \label{appB-eq13}
    \mu_i=\mu_i(\rho_i,\theta)=(\rho_i\psi_i)_{\rho_i}
\end{equation} \begin{equation} \label{appB-eq14}
    -\eta_i=(\psi_i)_\theta
\end{equation} \begin{equation} \label{appB-eq15}
    p_i=p_i(\rho_i,\theta)=-\rho_i\psi_i+\rho_i\mu_i
\end{equation}
This also has the implication that $\rho\psi=\sum_i\rho_i\psi_i$ and $\rho\eta=\sum_i\rho_i\eta_i$. For an extensive discussion on simple mixtures we refer to \cite{Mue}, while for special cases of simple mixtures to \cite[sec 15]{BDre}.


\section{Derivation of the Entropy Equation for the Type-II system}\label{appC}

For the derivation of equation \eqref{eq-entropytypeII}, one multiplies the mass equations \eqref{intro-eq1} by $\frac{1}{2}v_i^2-\mu_i$, the momentum equations \eqref{intro-eq2} by $-v_i$ and the energy equation \eqref{intro-eq3} by $1$ and sums them up to obtain, by making use of the definition $u_i=v_i-v$ and noticing that $u_i-u_j=v_i-v_j$ the following: \[ \begin{split}
    \del_t & (\rho\eta\theta)-\del_tp+\diver(\rho\eta\theta v)-\diver\left(\sum_i\rho_i\mu_iu_i\right)-\diver\left(\frac{1}{2}\sum_i\rho_iv_i^2u_i\right)-\diver(pv) \\
    & = -\sum_i\rho_i\del_t\mu_i-\sum_i\rho_iv_i\cdot\nabla\mu_i-\sum_i\rho_ib_i\cdot v_i+\sum_i\rho_i\nabla\mu_i\cdot v_i+\frac{\nabla\theta}{\theta}\sum_i(\rho_ie_i+p_i-\rho_i\mu_i)v_i \\
    & + \frac{\theta}{\epsilon}\sum_{i,j}b_{ij}\rho_i\rho_j(v_i-v_j)v_i+\diver(\kappa\nabla\theta)-\diver\left(\sum_i(\rho_ie_i+p_i)u_i\right)-\diver\left(\frac{1}{2}\sum_i\rho_iv_i^2u_i\right) \\
    & - \diver(pv) + \rho b\cdot v+\rho r+\sum_i\rho_ib_i\cdot u_i
\end{split} \] Now doing the necessary simplifications and using the Gibbs-Duhem relation, one obtains \[ \begin{split}
    \del_t & (\rho\eta\theta)+\diver(\rho\eta\theta v)=\diver\left(\kappa\nabla\theta-\sum_i(\rho_ie_i+p_i-\rho_i\mu_i)u_i\right)+\frac{\nabla\theta}{\theta}\sum_i(\rho_ie_i+p_i-\rho_i\mu_i)u_i \\
    & + \del_tp-\sum_i\rho_i\del_t\mu_i+\rho\eta v\cdot\nabla\theta+\frac{\theta}{\epsilon}\sum_{i,j}b_{ij}\rho_i\rho_j(v_i-v_j)v_i+\rho r
\end{split} \] If we differentiate the Gibbs-Duhem relation with respect to time we get \[ \del_tp=\rho\eta\del_t\theta+\sum_i\rho_i\del_t\mu_i \] thus, we divide by $\theta$ to obtain \[ \begin{split}
    \del_t (\rho\eta)+\diver(\rho\eta v) & = \diver\left(\frac{1}{\theta}\kappa\nabla\theta-\frac{1}{\theta}\sum_i(\rho_ie_i+p_i-\rho_i\mu_i)u_i\right) \\
    & + \frac{1}{\theta^2}\kappa|\nabla\theta|^2+\frac{1}{\epsilon}\sum_{i,j}b_{ij}\rho_i\rho_j(v_i-v_j)v_i +\frac{\rho r}{\theta}
\end{split} \] Finally, note that due to the symmetry of $b_{ij}$ and the fact that $u_i-u_j=v_i-v_j$ we find that \[ \frac{1}{\epsilon}\sum_{i,j}b_{ij}\rho_i\rho_j(v_i-v_j)v_i=\frac{1}{2\epsilon}\sum_{i,j}b_{ij}\rho_i\rho_j(v_i-v_j)^2=\frac{1}{2\epsilon}\sum_{i,j}b_{ij}\rho_i\rho_j(u_i-u_j)^2 \] which concludes the computation.


\section{Relative Entropy Identity}\label{appD}

In this appendix we present the computations leading to the derivation of \eqref{conv-eq15}. We first add \eqref{conv-eq13} and \eqref{conv-eq14} and use the formula \eqref{relen}
and 
\begin{equation} \label{appC-eq1}
  \begin{split}
        Q(U|\bar{U}) & =-\rho\eta v\bar{\theta}+\bar{\rho}\bar{\eta}\bar{v}\bar{\theta}+\sum_i\left(\frac{1}{2}\bar{v}^2-\bar{\mu}_i\right)(\rho_iv-\bar{\rho}_i\bar{v})-\bar{v}\rho v^2 \\
        & + \bar{v}\bar{\rho}\bar{v}^2-\bar{v}(p-\bar{p})+\left(\rho e+\frac{1}{2}\rho v^2+p\right)v-\left(\bar{\rho}\bar{e}+\frac{1}{2}\bar{\rho}\bar{v}^2+\bar{p}\right)\bar{v} \\
        & = vI(U|\bar{U})+(p-\bar{p})(v-\bar{v})
    \end{split}
\end{equation}
to arrive at
\begin{equation} \label{appC-eq2}
\del_tI(U|\bar{U})+\textnormal{div}Q(U|\bar{U})= T_1+T_2+T_3+T_4
\end{equation}
where \[ T_1:=\sum_i\del_t\left(\frac{1}{2}\bar{v}^2-\bar{\mu}_i\right)(\rho_i-\bar{\rho_i})-\del_t\bar{v}\cdot(\rho v-\bar{\rho}\bar{v})-\del_t\bar{\theta}(\rho\eta-\bar{\rho}\bar{\eta}) \]  \[ T_2:=\sum_i\nabla\left(\frac{1}{2}\bar{v}^2-\bar{\mu}_i\right)\cdot(\rho_iv-\bar{\rho}_i\bar{v})-\rho v\nabla\bar{v}\cdot v+\bar{\rho}\bar{v}\nabla\bar{v}\cdot\bar{v}-(p-\bar{p})\textnormal{div}\bar{v}+\nabla\bar{\theta}\cdot(-\rho\eta v+\bar{\rho}\bar{\eta}\bar{v}) \] \[ \begin{split}
    T_3 & := \textnormal{div}\left(-\frac{\bar{\theta}}{\theta}\kappa\nabla\theta\right)+\frac{1}{\theta}\kappa\nabla\theta\cdot\nabla\bar{\theta}-\frac{\bar{\theta}}{\theta^2}\kappa|\nabla\theta|^2+\textnormal{div}\left(\frac{\bar{\theta}}{\bar{\theta}}\bar{\kappa}\nabla\bar{\theta}\right)-\frac{1}{\bar{\theta}}\bar{\kappa}\nabla\bar{\theta}\cdot\nabla\bar{\theta} \\
    & + \frac{\bar{\theta}}{\bar{\theta}^2}\bar{\kappa}|\nabla\bar{\theta}|^2+\textnormal{div}\left(\kappa\nabla\theta\right)-\textnormal{div}(\bar{\kappa}\nabla\bar{\theta})
\end{split} \] and \[ \begin{split}
    T_4 & := \textnormal{div}\left(\frac{\bar{\theta}}{\theta}\sum_j(h_j-\rho_j\mu_j)u_j\right)-\frac{1}{\theta}\sum_j(h_j-\rho_j\mu_j)u_j\cdot\nabla\bar{\theta}+\frac{\bar{\theta}}{\theta}\sum_ju_j\cdot d_j \\
    & -\sum_i\left(\frac{1}{2}\bar{v}^2-\bar{\mu}_i\right)\textnormal{div}(\rho_iu_i)+\textnormal{div}\left(-\sum_jh_ju_j\right)
\end{split} \] Now \[ \begin{split}
    T_1 & = -\rho(v-\bar{v})\del_t\bar{v}+\del_t\bar{\theta}(-\rho\eta)(U|\bar{U})-(\bar{\rho}\bar{\eta})_\theta(\theta-\bar{\theta})\del_t\bar{\theta}-\sum_j\sum_i(\bar{\mu}_j)_{\rho_i}(\rho_i-\bar{\rho}_i)\del_t\bar{\rho}_j \\
    & =: T_{11}+T_{12}+T_{13}+T_{14}
\end{split} \] where we have used the fact that $(\mu_i)_\theta=-(\rho\eta)_{\rho_i}$ and introduced the relative quantity \[ (-\rho\eta)(U|\bar{U})=-\rho\eta+\bar{\rho}\bar{\eta}+\sum_j(\bar{\rho}\bar{\eta})_{\rho_j}(\rho_j-\bar{\rho}_j)+(\bar{\rho}\bar{\eta})_\theta(\theta-\bar{\theta}) \]   Using the entropy balance \eqref{conv-eq10} we get \[ \begin{split}
    T_{13} & = \del_t(-\bar{\rho}\bar{\eta})(\theta-\bar{\theta})+\sum_j(\bar{\rho}\bar{\eta})_{\rho_j}\del_t\bar{\rho}_j(\theta-\bar{\theta}) \\
    & = \textnormal{div}(\bar{\rho}\bar{\eta}\bar{v})(\theta-\bar{\theta})+\textnormal{div}\left(-\frac{1}{\bar{\theta}}\bar{\kappa}\nabla\bar{\theta}\right)(\theta-\bar{\theta})-\frac{1}{\bar{\theta}^2}\bar{\kappa}|\nabla\bar{\theta}|^2(\theta-\bar{\theta}) \\
    & - \sum_j(\bar{\rho}\bar{\eta})_{\rho_j}\nabla\bar{\rho}_j\cdot\bar{v}(\theta-\bar{\theta})-\sum_j(\bar{\rho}\bar{\eta})_{\rho_j}\bar{\rho}_j\textnormal{div}\bar{v}(\theta-\bar{\theta}) \\
    & =: T_{131}+T_{132}+T_{133}+T_{134}+T_{135}
\end{split} \] where \[ \begin{split}
    T_{131} & = \nabla(\bar{\rho}\bar{\eta})\cdot\bar{v}(\theta-\bar{\theta})+\bar{\rho}\bar{\eta}\textnormal{div}\bar{v}(\theta-\bar{\theta}) \\
    & =: T_{1311}+T_{1312}
\end{split} \] Furthermore \[ \begin{split}
    T_{14} & = \sum_j\sum_i(\bar{\mu}_{i})_{\rho_j}(\rho_i-\bar{\rho}_i)\nabla\bar{\rho}_j\cdot\bar{v}+\sum_j\sum_i(\bar{\mu}_i)_{\rho_j}(\rho_i-\bar{\rho}_i)\bar{\rho_j}\textnormal{div}\bar{v} \\
    & = \sum_i\nabla\bar{\mu}_i\cdot\bar{v}(\rho_i-\bar{\rho}_i)-\sum_i(\bar{\mu}_i)_\theta\nabla\bar{\theta}\cdot\bar{v}(\rho_i-\bar{\rho}_i)+\sum_j\sum_i(\bar{\mu}_i)_{\rho_j}(\rho_i-\bar{\rho}_i)\bar{\rho_j}\textnormal{div}\bar{v} \\
    & =: T_{141}+T_{142}+T_{143}
\end{split} \] Introducing the relative pressure \[ p(U|\bar{U})=p-\bar{p}-\sum_j\bar{p}_{\rho_j}(\rho_j-\bar{\rho}_j)-\bar{p}_\theta(\theta-\bar{\theta}) \] we get \[ \begin{split}
    T_2 & = -\rho(v-\bar{v})\nabla\bar{v}\cdot(v-\bar{v})-\sum_i\nabla\bar{\mu}_i\cdot(\rho_iv-\bar{\rho}_i\bar{v})-p(U|\bar{U})\textnormal{div}\bar{v}-\sum_j\bar{p}_{\rho_j}(\rho_j-\bar{\rho}_j)\textnormal{div}\bar{v} \\
    & - \bar{p}_\theta(\theta-\bar{\theta})\textnormal{div}\bar{v}+\bar{v}\cdot\nabla\bar{\theta}(-\rho\eta)(U|\bar{U})-\bar{v}\cdot\nabla\bar{\theta}\sum_j(\bar{\rho}\bar{\eta})_{\rho_j}(\rho_j-\bar{\rho}_j)-\bar{v}\cdot\nabla\bar{\theta}(\bar{\rho}\bar{\eta})_\theta(\theta-\bar{\theta}) \\
    & + \nabla\bar{\theta}(-\rho\eta)\cdot(v-\bar{v})-\rho\bar{v}\nabla\bar{v}\cdot(v-\bar{v}) \\
    & =: T_{21}+\cdots+T_{210}
\end{split} \] where \[ \begin{split}
    T_{22} & = -\sum_i\nabla\bar{\mu}_i\cdot(\rho_i-\bar{\rho}_i)v-\sum_i\nabla\bar{\mu}_i\cdot\bar{\rho}_i(v-\bar{v}) \\
    & = -\sum_i\nabla\bar{\mu}_i\cdot(\rho_i-\bar{\rho}_i)(v-\bar{v})-\sum_i\nabla\bar{\mu}_i\cdot(\rho_i-\bar{\rho}_i)\bar{v}+(v-\bar{v})\bar{\rho}\bar{\eta}\cdot\nabla\bar{\theta}-(v-\bar{v})\cdot\nabla\bar{p} \\
    & =: T_{221}+T_{222}+T_{223}+T_{224}
\end{split} \] Now, $T_{28}$ and $T_{134}$ cancel with $T_{1311}$ and the same holds for $T_{141}$ with $T_{222}$ and $T_{142}$ with $T_{27}$. Moreover \[ T_{29}+T_{223}=-(\rho\eta-\bar{\rho}\bar{\eta})(v-\bar{v})\cdot\nabla\bar{\theta} \] and by the momentum balance \eqref{conv-eq8} \[ \begin{split}
    T_{224} & = \frac{1}{\bar{\rho}}(\rho-\bar{\rho})(v-\bar{v})\cdot\nabla\bar{p}-\frac{\rho}{\bar{\rho}}(v-\bar{v})\cdot\nabla\bar{p} \\
    & = \frac{1}{\bar{\rho}}(\rho-\bar{\rho})(v-\bar{v})\cdot\nabla\bar{p}-\rho(v-\bar{v})\cdot\del_t\bar{v}+\rho\bar{v}\nabla\bar{v}\cdot(v-\bar{v}) \\
    & =: T_{2241}+T_{2242}+T_{2243}
\end{split} \] where $T_{2242}$ cancels with $T_{11}$ and $T_{2243}$ with $T_{210}$.

Differentiation of the Gibbs-Duhem relation \eqref{intro-gd} with respect to mass density and temperature, respectively,
gives
 \begin{equation} \label{appC-eq3}
 p_{\rho_i}=\sum_j\rho_j(\mu_j)_{\rho_i} \, ,  \qquad  p_\theta= \rho\eta+ \sum_j\rho_j(\mu_j)_\theta  \, . 
 \end{equation}
 Hence $T_{25},T_{1312}$ cancel with $T_{135}$ and $T_{24}$ with $T_{143}$. Finally, taking the gradient of the Gibbs-Duhem relation and using \eqref{appC-eq3}, we obtain the following identity \[ \nabla p=\rho\eta\nabla\theta+\sum_j\rho_j\nabla\mu_j \] which allows us to write 
\[ \begin{split}
    & - (\rho\eta-\bar{\rho}\bar{\eta})(v-\bar{v})\cdot\nabla\bar{\theta}-\sum_j\nabla\bar{\mu}_j\cdot\left(\rho_j-\bar{\rho}_j\right)(v-\bar{v})+\frac{1}{\bar{\rho}}(\rho-\bar{\rho})(v-\bar{v})\cdot\nabla\bar{p} 
\\
    & = -(\eta-\bar{\eta})\rho(v-\bar{v})\cdot\nabla\bar{\theta}-\sum_j\nabla\bar{\mu}_j\left(\frac{\rho_j}{\rho}-\frac{\bar{\rho}_j}{\bar{\rho}}\right)\cdot\rho(v-\bar{v})
\end{split} \] Therefore \begin{equation} \label{appC-eq4}
     \begin{split}
         T_1+T_2 & = (\del_t\bar{\theta}+\bar{v}\cdot\nabla\bar{\theta})(-\rho\eta)(\omega|\bar{\omega})-p(\omega|\bar{\omega})\textnormal{div}\bar{v}-(\eta-\bar{\eta})\rho(v-\bar{v})\cdot\nabla\bar{\theta} \\
    &\;\;  -\sum_j\nabla\bar{\mu}_j\left(\frac{\rho_j}{\rho}-\frac{\bar{\rho}_j}{\bar{\rho}}\right)\cdot\rho(v-\bar{v})-\rho(v-\bar{v})\nabla\bar{v}\cdot(v-\bar{v})+T_{132}+T_{133}
     \end{split}
 \end{equation} 
 
 Regarding $T_3$, we have \[ \begin{split}
    T_3 & = -\textnormal{div}\left(\bar{\theta}\frac{1}{\theta}\kappa\nabla\theta-\bar{\theta}\frac{1}{\bar{\theta}}\bar{\kappa}\nabla\bar{\theta}\right)+\textnormal{div}\left(\kappa\nabla\theta-\bar{\kappa}\nabla\bar{\theta}\right) \\
    & + \nabla\bar{\theta}\left(\frac{1}{\theta}\kappa\nabla\theta-\frac{1}{\bar{\theta}}\bar{\kappa}\nabla\bar{\theta}\right)-\bar{\theta}\left(\frac{1}{\theta^2}\kappa|\nabla\theta|^2-\frac{1}{\bar{\theta}^2}\bar{\kappa}|\nabla\bar{\theta}|^2\right)
\end{split} \] and thus \begin{equation} \label{appC-eq5}
    \begin{split}
    T_3+T_{132}+T_{133} & = \textnormal{div}\left[(\theta-\bar{\theta})\left(\frac{1}{\theta}\kappa\nabla\theta-\frac{1}{\bar{\theta}}\bar{\kappa}\nabla\bar{\theta}\right)\right]-\bar{\theta}\kappa\left(\frac{\nabla\theta}{\theta}-\frac{\nabla\bar{\theta}}{\bar{\theta}}\right)^2 \\
    & - \left(\frac{\nabla\theta}{\theta}-\frac{\nabla\bar{\theta}}{\bar{\theta}}\right)\frac{\nabla\bar{\theta}}{\bar{\theta}}(\bar{\theta}\kappa-\theta\bar{\kappa})
\end{split}
\end{equation} As for $T_4$, we have \begin{equation} \label{appC-eq6}
    \begin{split}
    T_4 & = -\textnormal{div}\Big[\sum_j\rho_ju_j(\mu_j-\bar{\mu}_j)+\frac{1}{\theta}(\theta-\bar{\theta})\sum_j(h_j-\rho_j\mu_j)u_j\Big]+\frac{\bar{\theta}}{\theta}\sum_ju_j\cdot d_j \\
    & -\sum_j\nabla\bar{\mu}_j\cdot\rho_ju_j-\frac{1}{\bar{\theta}}\nabla\bar{\theta}\sum_j(h_j-\rho_j\mu_j)u_j
\end{split}
\end{equation} Then, by \eqref{appC-eq1}-\eqref{appC-eq2} and \eqref{appC-eq4}-\eqref{appC-eq6} we obtain the relative entropy identity \eqref{conv-eq15}.

A more elegant, however not more advantageous way of writing \eqref{conv-eq15} is by noting that \begin{equation}
    (-\rho\eta)(\omega|\bar{\omega})\bar{v}\cdot\nabla\bar{\theta}-(\eta-\bar{\eta})\rho(v-\bar{v})\cdot\nabla\bar{\theta}=(-\rho\eta v)(U|\bar{U})\cdot\nabla\bar{\theta}+\bar{\eta}(\rho-\bar{\rho})(v-\bar{v})\cdot\nabla\bar{\theta}
\end{equation} Indeed, \[ \begin{split}
    (-\rho\eta v)(U|\bar{U})+\bar{\eta}(\rho-\bar{\rho})(v-\bar{v}) & = -\rho\eta v+\bar{\rho}\bar{\eta}\bar{v}+\sum_i(\bar{\rho}\bar{\eta})_{\rho_i}\bar{v}(\rho_i-\bar{\rho}_i)+\bar{\rho}\bar{\eta}(v-\bar{v}) \\
    & + (\bar{\rho}\bar{\eta})_\theta\bar{v}(\theta-\bar{\theta})+\bar{\eta}v\rho-\bar{\eta}v\bar{\rho}-\bar{\eta}\bar{v}\rho+\bar{\eta}\bar{v}\bar{\rho} \\
    & = (-\rho\eta)(\omega|\bar{\omega})\bar{v}+\rho\eta\bar{v}-\rho\eta v+\rho\bar{\eta}v-\rho\bar{\eta}\bar{v} \\
    & = (-\rho\eta)(\omega|\bar{\omega})\bar{v}-(\eta-\bar{\eta})\rho(v-\bar{v})
\end{split} \] in which case the right-hand side $\textnormal{RHS}$ of \eqref{conv-eq15} reads \[ \begin{split}
   \textnormal{RHS} =& \;  \del_t\bar{\theta}(-\rho\eta)(\omega|\bar{\omega})-\rho(v-\bar{v})\nabla\bar{v}(v-\bar{v})-p(\omega|\bar{\omega})\textnormal{div}\bar{v}+(-\rho\eta v)(U|\bar{U})\cdot \nabla\bar{\theta} \\
    & + \bar{\eta}(\rho-\bar{\rho})(v-\bar{v})\cdot\nabla\bar{\theta}-\sum_i\nabla\bar{\mu}_i\left(\frac{\rho_i}{\rho}-\frac{\bar{\rho}_i}{\bar{\rho}}\right)\rho(v-\bar{v})-\sum_i\nabla\bar{\mu}_i\cdot\rho_iu_i \\
    & - \left(\frac{\nabla\theta}{\theta}-\frac{\nabla\bar{\theta}}{\bar{\theta}}\right)\frac{\nabla\bar{\theta}}{\bar{\theta}}(\bar{\theta}\kappa-\theta\bar{\kappa})-\frac{1}{\bar{\theta}}\nabla\bar{\theta}\sum_i(h_i-\rho_i\mu_i)u_i.
\end{split} \]

\end{appendix}

\medskip
\noindent
{\bf Acknowledgement} We would like to thank the anonymous referees for their very helpful comments as well as bringing certain references to
our attention.

\end{document}